\documentclass[reqno, 11pt]{amsart}
\usepackage{pxfonts}
\usepackage{amsthm}
\usepackage{pgfplots}
\usepgfplotslibrary{fillbetween}
\usepackage{amsaddr} 

\theoremstyle{plain} 
\newtheorem{theorem}{Theorem}
\newtheorem{lemma}[theorem]{Lemma}
\newtheorem{prop}[theorem]{Proposition}
\newtheorem{coro}[theorem]{Corollary}

\newtheorem{oldtheorem}{Theorem}

\theoremstyle{definition} 
\newtheorem{defn}[theorem]{Definition}

\newtheorem{example}[theorem]{Example}

\theoremstyle{remark} 
\newtheorem{remark}[theorem]{Remark}

\numberwithin{equation}{section}
\numberwithin{theorem}{section}

\usepackage{tikz-cd}
\usetikzlibrary{arrows,arrows.meta}
\usetikzlibrary{calc}
\usepackage{caption}
\usepackage{subcaption}

\renewcommand{\d}{\mathrm d}
\renewcommand{\div}{\mathrm{div}}
\newcommand{\lap}{\Delta}
\newcommand{\Grad}{\nabla}
\newcommand{\bdy}{\partial}

\DeclareMathOperator{\supp}{supp}
\DeclareMathOperator{\diag}{diag}
\DeclareMathOperator{\sign}{sign}

\newcommand{\bb}{\mathbb}
\newcommand{\mc}{\mathcal}
\newcommand{\abs}[1]{\left|#1\right|}

\newcommand{\weakconv}{\rightharpoonup}
\newcommand{\lb}{\langle}
\newcommand{\rb}{\rangle}

\renewcommand{\epsilon}{\varepsilon}

\usepackage{enumitem} 
\usepackage[normalem]{ulem} 

\usepackage[colorlinks = true, citecolor = red]{hyperref}
\renewcommand{\d}{\mathrm d} 
\usepackage{bm}

\title[Anisotropic quasilinear systems]{Anisotropic Quasilinear Elliptic Systems with Homogeneous Critical Nonlinearities} 
\author{Mathew Gluck}
\address{Southern Illinois University \\
	School of Mathematical and Statistical Sciences}
\email{mathew.gluck@siu.edu}
\date{\today}

\begin{document} 
\flushbottom
\maketitle
\thispagestyle{empty}
\begin{abstract}
In this work we consider a system of quasilinear elliptic equations driven by an anisotropic $p$-Laplacian. The lower-order nonlinearities are in potential form and exhibit critical Sobolev growth. We exhibit conditions on the coefficients of the differential operator, the domain of the unknown function, and the lower-order nonlinearities under which nontrivial solutions are guaranteed to exist and conditions on these objects under which a nontrivial solution does not exist. 
\end{abstract}
\section{Introduction}
In 1983 Brezis and Nirenberg \cite{BrezisNirenberg1983} determined conditions on $\lambda\in \bb R$ and the bounded domain $\Omega \subset \bb R^n$ ($n\geq 3$) for which the problem 
\begin{equation}
\label{eq:BN_problem}
\begin{cases}
	-\lap u = \lambda u + |u|^{\frac{4}{n -2}}u & \text{ in }\Omega\\
	u= 0 & \text{ on }\bdy \Omega	
\end{cases}
\end{equation}
admits a positive solution and conditions on these objects under which problem \eqref{eq:BN_problem} does not admit a positive solution. They established the following theorem. In the statement of the theorem, $\lambda_1 = \lambda_1(-\lap)$ is the first eigenvalue of the Dirichlet Laplacian. 
\begin{oldtheorem}
\label{theorem:BN}
Let $n\geq 3$ and let $\Omega \subset \bb R^n$ be a bounded open set. 
\begin{enumerate}[label = \bf{(\alph*)}]
	\item If $n = 3$ then there are constants $\lambda_*(\Omega)\leq \lambda^*(\Omega)$ satisfying $0< \lambda_*\leq \lambda^*<\lambda_1$ such that problem \eqref{eq:BN_problem} admits a positive solution if $\lambda\in (\lambda^*, \lambda_1)$ and problem \eqref{eq:BN_problem} does not admit a positive solution if $\lambda\in (0, \lambda_*]$.
	\item If $n\geq 4$ then problem \eqref{eq:BN_problem} admits a positive solution if and only if $\lambda\in (0, \lambda_1)$. 
\end{enumerate}
\end{oldtheorem}
The subtleties under which problem \eqref{eq:BN_problem} is solvable, at least in the case that one in interested in positive solutions, are apparent in the statement of Theorem \ref{theorem:BN}. These subtleties make problem \eqref{eq:BN_problem} a natural candidate for further investigations and generalizations and there is a substantial body of mathematical research devoted to such pursuits. With regard to exploring extensions of problem \eqref{eq:BN_problem}, one theme present in the literature is concerned with replacing $-\lap$ in problem \eqref{eq:BN_problem} with a more general operator. Works in this family concern the extension to the $p$-Laplacian \cite{GarciaPeral1987, GueddaVeron1989}, the extension to an operator of Hardy type \cite{Jannelli1999,GazzolaGrunau2001,FerreroGazzola2001}, the extension to the fractional Laplacian \cite{ServadeiValdinoci2015} and the extension to more general divergence-form operators \cite{Egnell1988,Hadiji2006,Hadiji2007, deMouraMontenegro2014, MontenegrodeMoura2015, HaddadMontenegro2016}.  A second theme is the extension of problem \eqref{eq:BN_problem} to the setting of vector-valued unknown functions \cite{AlvesEtAl2000, Amster2002, BartschGuo2006, BrownEtAl2023}. In the vector-valued setting, a number of works have investigated the case where the analog of the right-hand side of the differential equation in problem \eqref{eq:BN_problem} is replaced by a sum of homogeneous functions \cite{Morais1999,BarbosaMontenegro2011,LuShen2020}. Such investigations are more interesting in the vector-valued setting than in the scalar-valued setting due to the fact that in the vector-valued setting there are many more homogeneous functions than there are in the scalar-valued setting.\\ 

In this work we consider a vector-valued analog of problem \eqref{eq:BN_problem} where the Laplacian is replaced by an anisotropic $p$-Laplacian and the right-hand side of the differential equation in \eqref{eq:BN_problem} is replaced by the sum of two (vector-valued) homogeneous functions, one of homogeneous degree $p - 1$ and the other of homogeneous degree $p^*- 1$, where $p^* = np/(n -p)$ is the critical exponent for the Sobolev embedding. To describe the problem of interest more specifically, we introduce the following notions of homogenity and anisotropic $p$-Laplacian. 
\begin{defn}
\label{defn:homogeneous_function}
A function $H:\bb R^d\to \bb R$ is homogeneous of degree $q$ if $H(\varrho s) = \varrho^qH(s)$ for all $(\varrho, s)\in (0, \infty)\times \bb R^d$.  We say $H$ is positively homogeneous of degree $q$ if $H$ is homogeneous of degree $q$ and $H(s)> 0$ for all $s\in \bb R^d\setminus\{0\}$. 
\end{defn}
Examples of homogeneous functions of degree $q$ on $\bb R^d$ include $H(s) = \sum_{j = 1}^dc_j|s_j|^q$ for $c_j\in \bb R$, $H(s) = |\lb Ms, s\rb|^{(q- 2)/2}\lb Ms, s\rb$ where $M$ is a $d\times d$ matrix and $\lb \cdot, \cdot \rb$ is the Euclidean inner product, $H(s) = \prod_{j = 1}^d|s_j|^{\alpha_j}$ where $\sum_{j = 1}^d \alpha_j = q$ and $H(s) = \pi_\ell(|s_1|, \ldots, |s_d|)^{q/\ell}$, where $\pi_\ell$ is the $\ell^{\text{th}}$ elementary symmetric polynomial in $d$ variables. \\

\noindent We define the anisotropic $p$-Laplacian $L_{A, p}$ by 
\begin{equation}
\label{eq:Lp_Laplacian}
\begin{split}
	L_{A, p}u = \div\left(\lb A(x)\Grad u, \Grad u\rb^{\frac{p - 2}{2}} A(x) \Grad u\right),
\end{split}
\end{equation}
where $p>1$ and $A$ is an $n\times n$ matrix-valued function. In this work we are concerned with establishing conditions under which either the existence or the non-existence of nontrivial solutions to 
\begin{equation}
\label{eq:main_problem}
\begin{cases}
	-L_{A, p}\bm u = f(\bm u) + g(\bm u) & \text{ in }\Omega\\
	\bm u = 0 & \text{ on }\bdy \Omega
\end{cases}
\end{equation}
is guaranteed, where $\bm u = (u_1, \ldots, u_d):\Omega \to \bb R^d$ is a vector-valued function and $L_{A,p}$ acts on $\bm u$ in the coordinatewise sense: 
\begin{equation*}
	L_{A, p}\bm u 
	= (L_{A, p}u_1, \ldots, L_{A,p }u_d). 
\end{equation*}
Problem \eqref{eq:main_problem} is assumed to be in potential form in the sense that the nonlinearities $f$ and $g$ satisfy
\begin{equation*}
	f = \frac1p\Grad F
	\qquad \text{ and }\qquad 
	g=\frac 1{p^*} \Grad G
\end{equation*} 
for some $p$-homogeneous function $F\in C^1(\bb R^d)$ and some $p^*$-homogeneous function $G\in C^1(\bb R^d)$. We do not consider positive solutions of problem \eqref{eq:main_problem} (solutions whose coordinate functions are non-negative, non-zero functions). Rather, we aim to establish results with general homogeneous nonlinearities. Under such generality, it is unclear whether the existence of a nontrivial solution implies the existence of a positive solution. See \cite{Morais1999, LuShen2020, BrownEtAl2023} for some restrictions on $F$ and $G$ under which such an implication holds. The results in the present work extend some results corresponding to some special cases of problem \eqref{eq:main_problem}. For example, \cite{Morais1999,BarbosaMontenegro2011} consider problem \eqref{eq:main_problem} with $A \equiv I$, \cite{Egnell1988, deMouraMontenegro2014, HaddadMontenegro2016} consider problem \eqref{eq:main_problem} in the scalar-valued setting with $A(x) = a(x)^{2/p}I$, and \cite{BrownEtAl2023} considers problem \eqref{eq:main_problem} with $p = 2$, for general $A$, and for particular choices of $F$ and $G$.\\

In order that $L_{A,p}$ retain the essential properties of the usual $p$-Laplacian we impose the following assumptions on $A$:
\begin{enumerate}[label = {\bf A\arabic*.}, ref = {\bf A\arabic*}]
	\item $A:\overline\Omega\to M(n; \bb R)$ is continuous \label{item:A_continuous}
	\item $A(x) = A^\top(x)$ for all $x\in \overline\Omega$ \label{item:A_symmetric}
	\item $A$ is uniformly positive definite in the sense that there is $\tau>0$ such that $
\tau|\xi|^2 \leq \lb A(x)\xi, \xi\rb$ for all $(x, \xi)\in \overline\Omega\times \bb R^n$, \label{item:A_positive_definite}
\end{enumerate}
where, in item \ref{item:A_continuous} and throughout the manuscript, $M(n; \bb R)$ denotes the space of $n\times n$ matrices with real entries. When $\Omega$ is bounded, assumptions \ref{item:A_continuous}, \ref{item:A_symmetric} and \ref{item:A_positive_definite} ensure the existence of a constant $\Lambda> 0$ for which 
\begin{equation}
\label{eq:A_upper_bounded}
	\lb A(x)\xi, \xi\rb \leq \Lambda |\xi|^2
	\qquad \text{ for all }(x, \xi)\in \overline\Omega\times \bb R^n. 
\end{equation}
The following smoothness and homogeneity assumptions will be imposed on the functions $F$ and $G$: 
\begin{enumerate}[label = {\bf H\arabic*.}, ref = {\bf H\arabic*}]
	\item $F\in C^1(\bb R^d)$ is homogeneous of degree $p$ \label{item:F_homogeneous}
	\item $G\in C^1(\bb R^d)$ is positively homogeneous of degree $p^*$. \label{item:G_positively_homogeneous}
\end{enumerate}
 In the spirit of the classical Brezis-Nirenberg problem \cite{BrezisNirenberg1983} we will establish the existence of solutions to problem \eqref{eq:main_problem} by showing that a suitable energy functional subject to a suitable constraint has sufficiently small infimum. In this context, it is natural to impose both a positivity assumption on $G$ and assumptions on $A$ and $F$ under which the energy functional is coercive. We collect and label these assumptions in \eqref{assumptions:AFG} below as they will be used in all of our existence theorems. To state the assumptions we define
\begin{equation}
\label{eq:minF_psphere}
	M_F = \max_{s\in \bb S_p^{d - 1}}F(s)
	\qquad 
	\text{ and }
	\qquad
	\mu_F= 	\min_{s\in \bb S_p^{d - 1}}F(s), 
\end{equation}
where 
\begin{equation*}
	\bb S_p^{d-1} 
	= \{s \in \bb R^d: |s_1|^p + \ldots + |s_d|^p = 1\}
\end{equation*}
and we let $\lambda_1(A, p)$ be the first eigenvalue of 
\begin{equation*}
	\begin{cases}
	-L_{A ,p} u = \lambda|u|^{p - 2}u & \text{ in }\Omega\\
	u = 0 & \text{ on }\bdy \Omega. 
	\end{cases}
\end{equation*}
The assumptions to be used in our existence theorems are 
\begin{equation}
\label{assumptions:AFG}\tag{AFG}
\begin{cases}
	A:\overline\Omega\to M(n; \bb R) \text{ satisfies \ref{item:A_continuous}, \ref{item:A_symmetric} and \ref{item:A_positive_definite}}, \\
	F \text{ satisfies both \ref{item:F_homogeneous} and }M_F\in (0, \lambda_1(A, p)), \text{ and }\\
	G \text{ satisfies \ref{item:G_positively_homogeneous}}. 
\end{cases}
\end{equation}

With these conventions in place we can state the main results of the present work. Generally, both the location of a global minimizer $x_0$ of $\det A$ and the behavior of $A$ near $x_0$ play roles in our sufficient conditions for existence of nontrivial solutions to problem \eqref{eq:main_problem}. The first of our existence theorems concerns the case where $A$ is not too flat near a global minimizer of $\det A$ and there is a positive lower bound on $\mu_F$. 
\begin{theorem}
\label{theorem:existence_gamma_small1}
Let $p\in (1, n)$, let $\Omega\subset \bb R^n$ be a bounded open set and suppose $A$, $F$ and $G$ satisfy assumptions \eqref{assumptions:AFG}. If there exists $x_0\in \Omega$ that minimizes $\det A$ and if there are constants $C_0>0$ and $\gamma\in (0, p]$ such that
\begin{equation}
\label{eq:A_lower_expansion}
	\lb A(x)\xi, \xi\rb^{p/2}
	\geq \lb A(x_0)\xi, \xi\rb^{p/2} + C_0|x - x_0|^\gamma |\xi|^p
	\qquad\text{ for all }(x, \xi)\in \Omega\times \bb R^n, 
\end{equation}
then there is $\lambda^*\in (0, \lambda_1(A,p))$ such that if $\mu_F> \lambda^*$ then problem \eqref{eq:main_problem} has a nontrivial weak solution. 
\end{theorem}
Our next existence theorem concerns the case where $A$ is sufficiently flat near a global minimizer of $\det A$. It does not require strict positivity of $F$ on $\bb R^d\setminus\{0\}$. 
\begin{theorem}
\label{theorem:existence_gamma_large}
Let $p\in (1, \sqrt n]$ and let $\Omega\subset \bb R^n$ be a bounded open set. Suppose $A$, $F$ and $G$ satisfy assumptions \eqref{assumptions:AFG} and suppose there is $s\in \bb S_p^{d-1}$ for which both 
\begin{equation}
\label{eq:good_theta}
	G(s) = \max\{G(t): t\in \bb S_p^{d -1}\}
	\qquad \text{ and }\qquad
	F(s)>0. 
\end{equation}
If there is $C_0>0$, $\gamma> p$, a minimizer $x_0\in\Omega$ of $\det A$ and $\delta>0$ such that
\begin{equation}
\label{eq:A_upper_expansion}
	\lb A(x)\xi, \xi\rb^{p/2}
	\leq \lb A(x_0), \xi, \xi\rb^{p/2} + C_0|x - x_0|^\gamma |\xi|^p 
\end{equation}
for all $(x, \xi)\in B(x_0, \delta)\times \bb R^n$, then problem \eqref{eq:main_problem} has a nontrivial weak solution. 
\end{theorem}
Theorem \ref{theorem:boundary_minimizer_existence} below concerns the case where a global minimizer $x_0$ of $\det A$ is on $\bdy \Omega$ and $\bdy\Omega$ has favorable geometry near $x_0$. This case was investigated in the scalar-valued setting and with $A(x) = a(x)^{2/p}I$ for some positive function $a$ by \cite{HaddadMontenegro2016}. To describe the meaning of favorable geometry, we give the following definition. 
\begin{defn}
The boundary of $\Omega$ is said to be \emph{interior $\theta$-singular} at $x_0\in \bdy\Omega$ with $\theta\geq 1$ if there is a constant $\delta>0$ and a sequence $(x^i)\subset\Omega$ such that $x^i\to x_0$ as $i\to\infty$ and $B(x^i, \delta|x^i - x_0|^\theta)\subset \Omega$. 
\end{defn}
\begin{example}
For $\theta\geq 1$, the set $\Omega = \{(x, y)\in \bb R^2: y> |x|^{1/\theta}\}$ is interior $\theta$-singular at the origin. Indeed, there is $\delta>0$ such that for any $0< r< 1$, the containment $B((0, r), \delta r^\theta)\subset\Omega$ holds. 
\end{example}
\begin{theorem}
\label{theorem:boundary_minimizer_existence}
Let $p\in (1, \sqrt n)$, let $\Omega \subset \bb R^n$ be a bounded open set and suppose $A$, $F$ and $G$ satisfy assumptions \eqref{assumptions:AFG}. Suppose there are $C_0>0$, 
\begin{equation}
\label{eq:gamma_largeness}
	\gamma > \max\left\{p^*, \frac{p(n - p)}{n - p^2}\right\},  
\end{equation}
a global minimizer $x_0\in\bdy\Omega$ for $\det A$ and $\delta>0$ such that \eqref{eq:A_upper_expansion} is satisfied for all $(x, \xi)\in (\Omega\cap B(x_0, \delta))\times \bb R^n$. If $\bdy\Omega$ is $\theta$ interior-singular at $x_0$ for some $1\leq \theta < \gamma\max\left\{p^*, \frac{p(n - p)}{n - p^2}\right\}^{-1}$ then problem \eqref{eq:main_problem} admits a nontrivial weak solution. 
\end{theorem}
Our final theorem is a nonexistence result. 
\begin{theorem}
\label{theorem:pozohaev_nonexistence}
Let $p\in (1, n)$, and let $\Omega \subset \bb R^n$ be a $C^1$ bounded domain that is star-shaped with respect to $x_0\in \Omega$. Let $A = (a_{ij})\in C^1(\overline\Omega \setminus\{x_0\}; M(n, \bb R))$ satisfy \ref{item:A_continuous}, \ref{item:A_symmetric} and \ref{item:A_positive_definite} and suppose $b_{ij}(x) = \lb x - x_0, \Grad a_{ij}(x)\rb$ extends continuously to $x_0$ for all $(i, j)\in \{1, \ldots, n\}\times\{1, \ldots, n\}$. Set $B(x) = (b_{ij}(x))$ and let $F$ and $G$ satisfy \ref{item:F_homogeneous} and \ref{item:G_positively_homogeneous} respectively. If there are $C_0> 0$ and $\gamma\in (0, p]$ for which 
\begin{equation}
\label{eq:B(x)_lower_bound}
	(x - x_0)\cdot \Grad(\lb A(x)\xi, \xi\rb^{p/2})
	\geq \gamma C_0|x - x_0|^\gamma |\xi|^p
	\qquad \text{for all } (x, \xi)\in \Omega\times \bb R^n,
\end{equation}
then there is a constant $\lambda_*> 0$ such that if $M_F< \lambda_*$ then problem \eqref{eq:main_problem} has no nontrivial solution $\bm u\in C^{1}(\overline\Omega)$. 
\end{theorem}
\begin{remark}
In \eqref{eq:B(x)_lower_bound}, the factor $\gamma$ in the multiplicative constant $\gamma C_0$ plays no essential role in the proof of Theorem \ref{theorem:pozohaev_nonexistence}. However, writing the multiplicative constant as such facilitates comparison between Theorems \ref{theorem:existence_gamma_small1} and \ref{theorem:pozohaev_nonexistence}. Indeed, an elementary computation shows that condition \eqref{eq:B(x)_lower_bound} implies condition \eqref{eq:A_lower_expansion}. Therefore, for sufficiently smooth $\Omega$ and sufficiently smooth $A$, if \eqref{eq:B(x)_lower_bound} holds for some minimizer $x_0$ of $\det A(x)$ relative to which $\Omega$ is star-shaped then combining Theorem \ref{theorem:existence_gamma_small1}, Theorem \ref{theorem:pozohaev_nonexistence} and the elliptic regularity theory we deduce the existence of constants $0< \lambda_*\leq \lambda^*< \lambda_1(A, p)$ such that a nontrivial $C^1$ solution to problem \eqref{eq:main_problem} exists whenever $\mu_F> \lambda^*$ and no such solution exists whenever $M_F< \lambda_*$. An example of a matrix-valued function for which the hypotheses of Theorem \ref{theorem:pozohaev_nonexistence} are satisfied, for which $x_0$ minimizes $\det A(x)$, and for which the regularity theory applies is $A(x) = M + C|x - x_0|^\gamma I$ where $M\in M(n; \bb R)$ is symmetric and positive definite and where $(\gamma, C)\in (0, p]\times(0, \infty)$. 
\end{remark}
This paper is organized as follows. In Section \ref{s:preliminaries} we discuss some preliminary notions including notational conventions, the functional space to be used in the variational argument, the first eigenvalue for $-L_{A, p}$ and properties of homogeneous functions. In Section \ref{s:sobolev_inequalities} we derive a sharp Sobolev-type inequality. The sharp constant in this inequality will be used in Section \ref{s:sufficient_condition_for_minimizer} to express a sufficient condition for the existence of a nontrivial solution to problem \eqref{eq:main_problem}. This condition is expressed as a smallness condition on the infimum of a suitably constrained energy functional. Theorems \ref{theorem:existence_gamma_small1}, \ref{theorem:existence_gamma_large} and \ref{theorem:boundary_minimizer_existence} (the existence theorems) are proved in Section \ref{s:proofs_of_existence} by exhibiting that this smallness condition is satisfied. In Section \ref{s:non_existence} we prove Theorem \ref{theorem:pozohaev_nonexistence} by using a variant of the well-known Pohozaev identity. Section \ref{s:appendix} is an appendix where we provide details of some computations that may be useful to some readers but whose inclusion in the main body of the manuscript would detract from the presentation. 
\section{Preliminaries}
\label{s:preliminaries}
Let $\Omega\subset \bb R^n$ be a bounded open set, let $p\in (1, n)$ and let $A:\overline\Omega\to M(n; \bb R)$ satisfy \ref{item:A_continuous}, \ref{item:A_symmetric} and \ref{item:A_positive_definite}. 
Define the norm $\|\cdot\|_{W_{0, A}^{1, p}(\Omega)}$ on $W_0^{1,p}(\Omega)$ by 
\begin{equation}
\label{eq:anisotropic_norm}
	\|u\|_{W_{0, A}^{1, p}(\Omega)}^p
	= \int_\Omega\lb A(x)\Grad u, \Grad u\rb^{p/2}\; \d x. 
\end{equation}
With this notation, in the case that $A\equiv I$ we recover the usual norm on $W_0^{1, p}(\Omega)$ in the sense that $\|u\|_{W_{0, I}^{1, p}(\Omega)} = \|u\|_{W_0^{1, p}(\Omega)}$ for all $u\in W_0^{1,p}(\Omega)$. When the context is clear we write $\|\cdot\|_{W_{0, A}^{1, p}}$ in place of $\|\cdot\|_{W_{0, A}^{1, p}(\Omega)}$.  The assumptions on $A$ guarantee that $\|\cdot\|_{W_{0, A}^{1, p}}$ is equivalent to the standard norm on $W_0^{1, p}$. In particular, denoting the completion of $C_c^\infty(\Omega)$ with respect to the norm $\|\cdot\|_{W_{0, A}^{1, p}(\Omega)}$ by $W_{0, A}^{1,p}(\Omega)$, we have $W_{0, A}^{1, p}(\Omega) = W_0^{1,p }(\Omega)$, where the equality is understood as an equality of sets. Despite this equivalence we retain the notational distinction between $\|\cdot\|_{W_0^{1, p}}$ and $\|\cdot\|_{W_{0, A}^{1, p}}$ as doing so will be convenient for expressing norms and inner products. For vector-valued functions $\bm u=(u_1, \ldots, u_d): \Omega \to \bb R^d$ we consider the product space
\begin{equation*}
	\mc W_A
	:= W_{0, A}^{1, p}(\Omega; \bb R^d)
	= W_{0, A}^{1, p}(\Omega)\times \ldots \times W_{0, A}^{1, p}(\Omega). 
\end{equation*}
The equivalence of the norms $\|\cdot\|_{W_0^{1, p}}$ and $\|\cdot\|_{W_{0, A}^{1, p}}$ ensures that as sets, $\mc W_A = \mc W_I = W_0^{1, p}(\Omega; \bb R^d)$. To ease notation, when making statements involving only set-theoretic notions (for example, set containment, membership in a set, etc.) we write $\mc W$ in place of $\mc W_A$ or $\mc W_I$. We will use the following norms for $\bb R^d$-valued functions: 
\begin{equation*}
\begin{split}
	\|\bm u\|_{L^p(\Omega; \bb R^d)}^p & = \|u_1\|_p^p +\ldots +  \|u_d\|_p^p\\
	\|\bm u\|_{\mc W_A}^p & := \|\bm u\|_{W_{0, A}^{1, p}(\Omega; \bb R^d)}^p = \|u_1\|_{W_{0, A}^{1, p}}^p + \ldots +\|u_d\|_{W_{0, A}^{1, p}}^p.  
\end{split}
\end{equation*}
As in the scalar-valued setting we retain the notational distinction between the norms $\|\cdot\|_{\mc W_I}$ and $\|\cdot \|_{\mc W_A}$ despite the equality of the sets $\mc W_I = \mc W = \mc W_A$. \\

From assumption \ref{item:A_symmetric} the operator $-L_{A, p}$ may be viewed as the differential of the map $u\mapsto\frac 1p\|u\|_{W_{0, A}^{1, p}(\Omega)}^p$ in the sense that 
\begin{equation*}
	\lb-L_{A, p}u, \varphi\rb_{W^{-1, p'}} 
	= \int_\Omega\lb A(x)\Grad u, \Grad u\rb^{\frac{p - 2}{2}}\lb A(x)\Grad u, \Grad \varphi\rb\; \d x
\end{equation*}
for all $u, \varphi\in W_0^{1,p}$, where $\lb\cdot, \cdot \rb_{W^{-1, p'}}$ is the dual pairing of $W_0^{1, p}$ with $W^{-1,p'}$. Correspondingly, a weak solution to problem \eqref{eq:main_problem} is a function $\bm u\in \mc W$ for which 
\begin{equation*}
	-L_{A, p}u_j = f_j(\bm u) + g_j(\bm u)
	\qquad \text{ for all }j\in \{1, \ldots, d\}, 
\end{equation*}
where the equalities are understood to hold in $W^{-1,p'}$. Although our existence theorems for problem \eqref{eq:main_problem} are formulated in terms of weak solutions, the following regularity result illustrates the fact that under sufficient smoothness assumptions on $A$ and $\Omega$, the regularity of weak solutions can be improved. The proof is standard except for the initial integrability boosting step. For the convenience of the reader, we overview the proof in Subsection \ref{ss:regularity} of the appendix. 
\begin{prop}
\label{prop:regularity}
Let $p\in (1, n)$, let $\alpha\in (0, 1)$ and suppose $\Omega \subset \bb R^n$ is a bounded domain with $\bdy\Omega \in C^{1, \alpha}$. Assume the entries of $A:\overline\Omega \to M(n; \bb R)$ are in $C^\alpha(\overline\Omega)$ and that $A$ satisfies both \ref{item:A_symmetric} and \ref{item:A_positive_definite}. Let $f = \frac 1p \Grad F$ and $g = \frac{1}{p^*}\Grad G$ for some functions $F$ and $G$ satisfying \ref{item:F_homogeneous} and \ref{item:G_positively_homogeneous} respectively. If $\bm u\in \mc W$ is a weak solution to problem \eqref{eq:main_problem} then $u\in C^{1, \alpha}(\overline\Omega)$. 
\end{prop}
\subsection{The Variational Problem}
\label{ss:variational_problem}
In this subsection we introduce a constrained minimization problem whose solutions (if they exist) yield nontrivial weak solutions to problem \eqref{eq:main_problem}. We start by recording the following lemma which lists some basic properties of homogeneous functions that will be used throughout the manuscript. Since the proof is straight-forward we omit the details.
\begin{lemma}[Properties of homogeneous functions]
\label{lemma:properties_homogeneous_functions}
If $H:\bb R^d\to \bb R$ is homogeneous of degree $q>0$ then the following properties hold: 
\begin{enumerate}[label = {\bf (\alph*)}, ref = {\bf(\alph*)}]
	\item $H(0) = 0$. 
	\item If $H\in C^1(\bb R^d)$ then for any $j\in\{1, \ldots, d\}$, the partial derivative $\partial_jH\in C^0(\bb R^d)$ is homogeneous of degree $q- 1$. \label{item:derivatives_homogeneous}
	\item If $H\in C^1(\bb R^d)$ then for every $s\in \bb R^d$ we have $\lb s, \Grad H(s)\rb = qH(s)$. \label{item:homogeneous_grad_dot}
	\item If $H\in C^0(\bb R^d)$ then for any $\alpha \in [1, \infty)$ there is $s\in \bb S_\alpha^{d - 1}$ such that $H(s) = M_H(\alpha) = \max\{H(t): t\in \bb S_\alpha^{d -1}\}$. Moreover, the inequality $H(s) \leq M_H(\alpha)|s|_\alpha^q$ holds for all $s\in \bb R^d$. \label{item:homogeneous_upper_bound}
\end{enumerate}
\end{lemma}
In Lemma \ref{lemma:properties_homogeneous_functions} \ref{item:homogeneous_upper_bound}, when $\alpha = p$ and $p$ is as in the definition of $L_{A,p}$ we write $M_H$ in place of $M_H(p)$. This notation is consistence with the notation in \eqref{eq:minF_psphere}. Next we introduce both the energy functional $\Phi_{A, F}:\mc W\to \bb R$ given by 
\begin{equation}
\label{eq:intro_Phi_AF}
	\Phi_{A, F}(\bm u)
	= \|\bm u\|_{\mc W_A}^p  - \int_\Omega F(\bm u)\;\d x
\end{equation}
and the constraint 
\begin{equation}
\label{eq:the_constraint}
	\mc M 
	= \{\bm u\in \mc W: \int_\Omega G(\bm u)\; \d x = 1\}. 
\end{equation} 
If $F$ and $G$ satisfy \ref{item:F_homogeneous} and \ref{item:G_positively_homogeneous} respectively and if $\Phi_{A, F}$ is coercive then the restriction $\Phi_{A, F}\big|_{\mc M}$ is bounded below. In this case, up to a positive constant multiple, any constrained minimizer of $\Phi_{A, F}\big|_{\mc M}$ is a weak solution to problem \eqref{eq:main_problem}. In what follows we formulate, under some hypotheses on $A$ and $F$, an equivalent condition for the coercivity of $\Phi_{A, F}$. 

From assumption \ref{item:A_positive_definite} we have
\begin{equation}
\label{eq:LAp_eigenquotient}
	\frac{\|u\|_{W_{0, A}^{1, p}}^p}{\|u\|_p^p}
	\geq \tau^{p/2}\frac{\|\Grad u\|_p^p}{\|u\|_p^p}
	\geq \tau^{p/2}\lambda_1(p)
	\qquad \text{ for all }u\in W_0^{1, p}(\Omega)\setminus\{0\}, 
\end{equation}
where $\lambda_1(p)>0$ is the first eigenvalue of $-\lap_p$ on $W_0^{1,p }(\Omega)$. Therefore, $-L_{A, p}$ has a first eigenvalue $\lambda_1(A, p)$ and this eigenvalue admits the variational characterization
\begin{equation}
\label{eq:lambda1_A}
	\lambda_1(A, p)
	= \inf\{\|u\|_{W_{0, A}^{1, p}}^p: u\in W_0^{1, p}(\Omega) \text{ and }\|u\|_p = 1\}>0,  
\end{equation}
where the strict positivity follows from \eqref{eq:LAp_eigenquotient} and the strict positivity of $\lambda_1(p)$. A standard argument based on the compactness of the embedding $W_0^{1, p}\hookrightarrow L^p$ shows that the infimum in the definition of $\lambda_1(A, p)$ is attained by some $u\in W_0^{1, p}$ for which $\|u\|_p = 1$. In fact, $\lambda_1(A, p)$ admits a strictly positive first eigenfunction as can be seen by combining the fact that $\||u|\|_{W_{0, A}^{1, p}} = \|u\|_{W_{0, A}^{1, p}}$ for all $u\in W_0^{1, p}(\Omega)$ with Trudinger's Harnack inequality \cite{Trudinger1967}. 

A simple construction shows that if $F\in C^0(\bb R^d)$ is a degree-$p$ homogeneous function for which $M_F>0$, then there is $\bm u\in \mc W$ for which $\int_\Omega F(\bm u)\; \d x> 0$, see Lemma \ref{lemma:simple_construction} in the appendix for details. For any such $F$, and for any $\bm u\in \mc W$ for which $\int_\Omega F(\bm u)\; \d x> 0$, the inequality 
\begin{equation*}
\begin{split}
	\int_\Omega F(\bm u)\; \d x
	& \leq M_F\int_\Omega|\bm u|_p^p \; \d x\\
	& \leq \lambda_1(p)^{-1}M_F\sum_{j = 1}^d \int_\Omega |\Grad u_j|^p\; \d x\\
	& \leq \tau^{-p/2}\lambda_1(p)^{-1}M_F\|\bm u\|_{\mc W_A}^p 
\end{split}
\end{equation*}
guarantees that the quantity
\begin{equation}
\label{eq:lambda1_FA}
	\lambda_{1, F}(A, p)
	= \inf\{\|\bm u\|_{\mc W_A}^p: \bm u\in \mc W \text{ and }\int_\Omega F(\bm u)\; \d x = 1\} 
\end{equation}
is well-defined. The following lemma characterizes $\lambda_{1, F}(A, p)$ as the smallest eigenvalue for a suitable eigenproblem. Since the proof follows from a routine argument involving Rellich's Theorem, we omit the details. 
\begin{lemma}
Let $p\in (1,n)$ and let $\Omega \subset \bb R^n$ be a bounded domain. If $A$ satisfies \ref{item:A_continuous}, \ref{item:A_symmetric}, \ref{item:A_positive_definite} and if $F$ satisfies both \ref{item:F_homogeneous} and $M_F>0$, then  $\lambda_{1, F}(A, p)$ as defined in \eqref{eq:lambda1_FA} is the smallest positive eigenvalue for the problem 
\begin{equation}
\label{eq:LAp_f_eigenproblem}
\begin{cases}
	-L_{A, p}\bm u = \lambda f(\bm u) & \text{ in }\Omega\\
	\bm u = 0 & \text{ on }\bdy\Omega, 
\end{cases}
\end{equation}
where $f = \frac{\Grad F}{p}$. 
\end{lemma}
The following lemma relates the values of $\lambda_1(A, p)$ and $\lambda_{1, F}(A, p)$ defined in \eqref{eq:lambda1_A} and \eqref{eq:lambda1_FA} respectively. 
\begin{lemma}
\label{lemma:relate_lambda1_to_lambda1F}
Let $p\in (1, n)$, let $\Omega\subset \bb R^n$ be a bounded open set and let $A$ satisfy \ref{item:A_continuous}, \ref{item:A_symmetric} and \ref{item:A_positive_definite}. If $F\in C^0(\bb R^d)$ is homogeneous of order $p$ and $M_F>0$, then the first eigenvalues $\lambda_1(A, p)$ and $\lambda_{1, F}(A, p)$ defined in \eqref{eq:lambda1_A} and \eqref{eq:lambda1_FA} respectively are related by 
\begin{equation*}
	M_F\lambda_{1, F}(A, p) = \lambda_1(A, p). 
\end{equation*}
\end{lemma}
\begin{proof}
For any $\bm u\in \mc W$ for which $\int_\Omega F(\bm u)\; \d x = 1$, Lemma \ref{lemma:properties_homogeneous_functions} \ref{item:homogeneous_upper_bound} gives $1\leq M_F\|\bm u\|_p^p$, from which we routinely obtain 
\begin{equation*}
	\|\bm u\|_{\mc W_A}^p
	\geq \frac{\|\bm u\|_{\mc W_A}^p}{M_F\|\bm u\|_p^p}
	\geq M_F^{-1}\lambda_1(A, p). 
\end{equation*}
Taking the infimum over all $\bm u\in \mc W$ satisfying $\int_\Omega F(\bm u)\; \d x = 1$ gives $M_F\lambda_{1, F}(A, p)\geq \lambda_1(A, p)$. To show the reverse inequality, let $s = (s_1, \ldots, s_d)\in \bb S_p^{d - 1}$ satisfy $M_F = F(s)$ and let $\varphi\in W_0^{1, p}(\Omega)$ be a first eigenfunction for $\lambda_1(A, p)$. Setting $\bm u = (s_1\varphi, \ldots, s_d\varphi)$ we have $\|\bm u\|_{\mc W_A} = \|\varphi\|_{W_{0, A}^{1, p}}$ and 
\begin{equation*}
	\int_\Omega F(\bm u)\; \d x
	= M_F \int_\Omega |\varphi|^p\; \d x
	> 0.
\end{equation*}
Since $\varphi$ is assumed to attain the infimum in \eqref{eq:lambda1_A} we get
\begin{equation*}
	\lambda_{1, F}(A, p)
	\leq \frac{\|\bm u\|_{\mc W_A}^p}{\int_\Omega F(\bm u)\;\d x}
	= \frac{\|\varphi\|_{W_{0, A}^{1, p}}^p}{M_F\|\varphi\|_p^p}
	= M_F^{-1}\lambda_1(A, p). 
\end{equation*}
\end{proof}
To close this section we establish the following equivalent condition for coercivity of $\Phi_{A, F}$. 
\begin{lemma}
\label{lemma:Phi_AF_coercive}
Let $p\in (1, n)$, let $\Omega \subset \bb R^n$ be a bounded open set and suppose $A$ satisfies \ref{item:A_continuous}, \ref{item:A_symmetric} and \ref{item:A_positive_definite}. If  $F\in C^0(\bb R^d)$ is homogeneous of degree $p$ and if $M_F> 0$ then the functional $\Phi_{A, F}$ as defined in \eqref{eq:intro_Phi_AF} is coercive if and only if $M_F< \lambda_1(A, p)$. 
\end{lemma}
\begin{proof}
If $\Phi_{A, F}$ is coercive then there is $\epsilon\in (0, 1)$ such that the inequality 
\begin{equation*}
	\|\bm u\|_{\mc W_A}^p - \int_\Omega F(\bm u)\; \d x
	\geq \epsilon\|\bm u\|_{\mc W_A}^p
\end{equation*}
holds for all $\bm u\in \mc W$. In particular, for any such $\epsilon$ and for any $\bm u\in \mc W$ for which $\int_\Omega F(\bm u)\; \d x > 0$ we have 
\begin{equation*}
	1\leq (1 - \epsilon)\frac{\|\bm u\|_{\mc W_A}^p}{\int_\Omega F(\bm u)\; \d x}. 
\end{equation*}
Taking the infimum over all such $\bm u$ and in view of the positivity of $\lambda_{1, F}(A, p)$ and Lemma \ref{lemma:relate_lambda1_to_lambda1F} we have 
\begin{equation*}
	1
	\leq ( 1- \epsilon)\lambda_{1, F}(A, p)
	< \lambda_{1, F}(A, p)
	= M_F^{-1}\lambda_1(A, p). 
\end{equation*}
This shows that the inequality $M_F< \lambda_1(A, p)$ is a necessary condition for the coercivity of $\Phi_{A, F}$. Next we proceed to show that this inequality is also sufficient for the coercivity of $\Phi_{A, F}$. For all $\bm u\in \mc W$ for which $\int_\Omega F(\bm u)\; \d x\leq 0$, the inequality $\Phi_{A, F}(\bm u)\geq \|\bm u\|_{\mc W_A}^p$ holds by inspection of $\Phi_{A, F}$. For $\bm u\in \mc W\setminus\{0\}$ satisfying $\int_\Omega F(\bm u)\; \d x> 0$ we have
\begin{equation*}
\begin{split}
	\Phi_{A, F}(\bm u)
	& = \left( 1- \frac{\int_\Omega F(\bm u)\; \d x}{\|\bm u\|_{\mc W_A}^p}\right)\|\bm u\|_{\mc W_A}^p\\
	& \geq \left(1 - \frac{1}{\lambda_{1, F}(A, p)}\right)\|\bm u\|_{\mc W_A}^p\\
	& = \left(1 - \frac{M_F}{\lambda_1(A, p)}\right)\|\bm u\|_{\mc W_A}^p,  
\end{split}
\end{equation*}
where the final equality holds by Lemma \ref{lemma:relate_lambda1_to_lambda1F}. We conclude that the inequality $M_F< \lambda_1(A,p)$ guarantees the coercivity of $\Phi_{A, F}$. 
\end{proof}
%
\section{Inequalities of Sobolev Type}
\label{s:sobolev_inequalities}
In Section \ref{s:sufficient_condition_for_minimizer}, we provide a sufficient condition for the existence of nontrivial solutions to problem \eqref{eq:main_problem}. This condition is expressed as a smallness threshold for the infimum of the energy functional $\Phi_{A, F}$ constrained to $\mc M$ and this threshold is quantified in terms of a sharp Sobolev type constant. The purpose of this section is to introduce and compute the value of this sharp constant.

For $p\in (1, n)$ and an open set $\Omega\subset \bb R^n$ (not necessarily bounded), the sharp constant in the classical Sobolev inequality is
\begin{equation}
\label{eq:classical_sobolev_constant}
	\mc S^{-1} 
	= \inf\{\|\Grad u\|_{L^p(\Omega)}^p: \|u\|_{L^{p^*}(\Omega)} = 1\}. 
\end{equation}
It is well-known that $\mc S^{-1}$ depends only on $n$ and $p$ \cite{Talenti1976} and that the infimum in the definition of $\mc S^{-1}$ is only attained when $\Omega = \bb R^n$. Moreover, for $\Omega = \bb R^n$ and for any $(\epsilon, x_0)\in (0, \infty)\times \bb R^n$, the infimum is attained by 
\begin{equation}
\label{eq:shifted_scaled_bubble}
	U_{\epsilon, x_0}(x)
	= \epsilon^{-\frac{n -p}{p}}U\left(\frac{x - x_0}{\epsilon}\right),
\end{equation}
where
\begin{equation}
\label{eq:normalized_bubble}
	U(x) = c_{n,p}\left(1 + |x|^{\frac{p}{p - 1}}\right)^{-\frac{n - p}{p}}
\end{equation}
and $c_{n, p}>0$ is chosen so that $\|U\|_{L^{p^*}(\bb R^n)} = 1$. A direct computation shows that $\|U_{\epsilon, x_0}\|_{L^{p^*}(\bb R^n)} = 1$ for all $(\epsilon, x_0)\in (0, \infty)\times \bb R^n$. For symmetric positive definite matrices $M\in M(n, \bb R)$ we introduce the anisotropic Sobolev constant $\mc S(M)$ defined by
\begin{equation}
\label{eq:M_sobolev_constant}
\begin{split}
	\mc S(M)^{-1}
	= \inf\left\{\|u\|_{W_{0, M}^{1, p}(\Omega)}^p: \|u\|_{L^{p^*}(\Omega)} = 1\right\}, 
\end{split}
\end{equation}
where $\|\cdot \|_{W_{0, M}^{1, p}(\Omega)}$ is as in \eqref{eq:anisotropic_norm}. 
With this notation, $\mc S^{-1} = \mc S(I)^{-1}$ is the sharp Sobolev constant in \eqref{eq:classical_sobolev_constant}. The following lemma relates the values of $\mc S$ and $\mc S(M)$. The proof follows from a routine computation using the change of variable $x\mapsto P^{-1}y$, where $P\in GL(n; \bb R)$ satisfies $PMP^\top = I_n$. See Appendix A of \cite{MontenegrodeMoura2015} for details. 
\begin{lemma}
\label{eq:constat_matrix_sobolev_constant}
Let $p\in (1, n)$ and let $\Omega \subset \bb R^n$ be an open set. If $M\in M(n;\bb R)$ is a symmetric positive definite matrix then the constants $\mc S$ and $\mc S(M)$ defined in \eqref{eq:classical_sobolev_constant} and \eqref{eq:M_sobolev_constant} are related via 
\begin{equation*}
	\mc S(M)= (\det M)^{-\frac{p}{2n}}\mc S. 
\end{equation*}
\end{lemma}
For vector-valued functions $\bm u = (u_1, \ldots, u_d)$ and for any $G\in C^0(\bb R^d)$ that is positively homogeneous of degree $p^*$ we have
\begin{equation*}
	\int_\Omega G(\bm u)\; \d x
	\leq 	M_G\int_\Omega \left(\sum_{j = 1}^d|u_j|^p\right)^{p^*/p}\; \d x, 
\end{equation*}
where $M_G = \max \{G(s): s\in \bb S_p^{d- 1}\}$ as in \eqref{eq:minF_psphere}. From this estimate, from Minkowski's inequality and from the definition of $\mc S(M)$ we have
\begin{equation}
\label{eq:as_in_this}
\begin{split}
	\left(\int_\Omega G(\bm u)\; \d x\right)^{p/p^*}
	& \leq M_G^{p/p^*}\sum_{j = 1}^d\left(\int_\Omega |u_j|^{p^*}\; \d x\right)^{p/p^*}\\
	& \leq \mc S(M)M_G^{p/p^*}\|\bm u\|_{\mc W_M}^p. 
\end{split}
\end{equation}
This computation shows that the quantity 
\begin{equation}
\label{eq:MG_sobolev_constant}
	\mc S(M; G)^{-1}
	= \inf\{\|\bm u\|_{\mc W_M}^p: \bm u\in \mc W \text{ and }\int_\Omega G(\bm u)\; \d x = 1\}
\end{equation}
is well-defined, positive and satisfies $\mc S(M; G)\leq M_G^{p/p^*}\mc S(M)$. The next lemma shows that equality holds in this inequality.  
\begin{lemma}
\label{eq:constat_matrix_G_sobolev_constant}
Let $p\in (1, n)$, let $\Omega \subset \bb R^n$ be an open set and let $M\in M(n, \bb R)$ be symmetric and positive definite. If $G\in C^0(\bb R^d)$ is positively homogeneous of degree $p^*$ then the Sobolev-type constants $\mc S(M)$ and $\mc S(M; G)$ defined in \eqref{eq:M_sobolev_constant} and \eqref{eq:MG_sobolev_constant} respectively are related via $\mc S(M; G) = M_G^{p/p^*}\mc S(M)$. 
\end{lemma}
\begin{proof}
The inequality $M_G^{p/p^*}\mc S(M)\geq \mc S(M; G)$ was established in the computation preceding the statement of the lemma. To show the reverse inequality, let $(\varphi^k)_{k = 1}^\infty\subset W_0^{1, p}(\Omega)$ be a sequence of nonnegative functions for which
\begin{equation*}
	\mc S(M)^{-1} + \circ(1)
	= \frac{\|\varphi^k\|_{W_{0, M}^{1, p}(\Omega)}^p}{\|\varphi^k\|_{L^{p^*}(\Omega)}^p}. 
\end{equation*}
Choose $s = (s_1, \ldots, s_d)\in \bb S_p^{d - 1}$ for which $G(s) = M_G$ and set $\bm u^k = (s_1\varphi^k, \ldots, s_d\varphi^k)$. Evidently we have both 
\begin{equation*}
	\int_\Omega G(\bm u^k)\; \d x
	= M_G\|\varphi^k\|_{L^{p^*}(\Omega)}^{p^*}
\end{equation*}
and $\|\bm u^k\|_{\mc W_M}^p= \|\varphi^k\|_{W_{0, M}^{1, p}(\Omega)}^p$. Therefore, 
\begin{equation*}
	\frac{\|\bm u^k\|_{\mc W_M}^p}{\left(\int_\Omega G(\bm u^k)\; \d x\right)^{p/p^*}}
	= M_G^{-p/p^*} \frac{\|\varphi^k\|_{W_{0, M}^{1, p}(\Omega)}^p}{\|\varphi^k\|_{L^{p^*}(\Omega)}^p}
	= M_G^{-p/p^*}\mc S(M)^{-1} + \circ(1). 
\end{equation*}
The inequality $M_G^{p/p^*}\mc S(M)\leq \mc S(M; G)$ follows immediately. 
\end{proof}
\begin{remark}
\label{remark:scaling_argument}
A simple rescaling argument shows that for all $(x_0, \delta)\in \overline\Omega \times (0,\infty)$, the quantity
\begin{equation*}
	\mc S_{x_0, \delta}^{-1}(M; G)
	:= \inf\{\|\bm u\|_{\mc W_M}^p: \bm u\in \mc W(x_0;  \delta) \text{ and }\int_\Omega G(\bm u)\; \d x = 1\}
\end{equation*}
satisfies $\mc S_{x_0, \delta}(M; G) = \mc S(M; G)$, where 
\begin{equation}
\label{eq:W_x0_delta}
	\mc W(x_0; \delta)
	:= W_0^{1, p}(\Omega\cap B(x_0, \delta); \bb R^d).
\end{equation}
\end{remark}
For non-constant coefficient matrices $A$, in place of an inequality of the form $\left(\int_\Omega G(\bm u)\; \d x\right)^{p/p^*}\leq C\|\bm u\|_{\mc W_A}^p$ as in \eqref{eq:as_in_this}, we consider the following inequality with lower-order term: 
\begin{equation}
\label{eq:sobolev_with_lower_order}
	\left(\int_\Omega G(\bm u)\; \d x\right)^{p/p^*}
	\leq C_1\|\bm u\|_{\mc W_A}^p + C_2\|\bm u\|_{L^p(\Omega; \bb R^d)}^p.
\end{equation}
In place of the sharp Sobolev type constant $\mc S(M; G)^{-1}$ defined in \eqref{eq:MG_sobolev_constant} we define 
\begin{equation*}
	\mc N(A; G)
	= \inf\{C_1: \text{there exists $C_2>0$ such that \eqref{eq:sobolev_with_lower_order} holds for all $\bm u\in \mc W_A$}\}.
\end{equation*}
The following proposition gives the explicit value of $\mc N(A; G)$ in terms of $A$, $G$ and the classical sharp Sobolev constant in \eqref{eq:classical_sobolev_constant}. 
\begin{prop}
\label{prop:N(A,G)_value}
Let $p\in (1, n)$, let $\Omega\subset \bb R^n$ be a bounded open set and let $A$ satisfy \ref{item:A_continuous}, \ref{item:A_symmetric} and \ref{item:A_positive_definite}. If $G\in C^0(\bb R^d)$ is positively homogeneous of degree $p^*$ then 
\begin{equation*}
	\mc N(A; G) 
	= m_A^{-\frac p{2n}}M_G^{p/p^*}\mc S,
\end{equation*} 
where 
\begin{equation}
\label{eq:mA_notation}
	m_A = \min\{\det A(x): x\in \overline\Omega\}.
\end{equation} 
\end{prop}
\begin{proof}[Proof of Proposition \ref{prop:N(A,G)_value}]
Define
\begin{equation*}
	\overline {\mc N}(A; G)
	= \max\{\mc S(A(x_0); G): x_0\in\overline\Omega\} 
\end{equation*}
and observe that combining Lemma \ref{eq:constat_matrix_sobolev_constant} and Lemma \ref{eq:constat_matrix_G_sobolev_constant} gives $\overline {\mc N}(A; G)
= m_A^{-\frac p{2n}}M_G^{p/p^*}\mc S$. In particular, the equality asserted by Proposition \ref{prop:N(A,G)_value} is precisely the equality $\mc N(A; G) = \overline{\mc N}(A; G)$. This equality follows by combining Lemma \ref{lemma:N_bar_upper_bound} and Lemma \ref{lemma:Nbar_epsilon_sharp_inequality} below where the inequalities $\mc N(A; G) \geq \overline{\mc N}(A; G)$ and $\mc N(A; G) \leq \overline{\mc N}(A; G)$ are established respectively. 
\end{proof}
\begin{lemma}
\label{lemma:N_bar_upper_bound}
Under the hypotheses of Proposition \ref{prop:N(A,G)_value}, if $C_1$ and $C_2$ are positive constants for which \eqref{eq:sobolev_with_lower_order} holds for all $\bm u\in \mc W$ then $C_1\geq \overline{\mc N}(A; G)$. In particular, $\overline{\mc N}(A; G) \leq \mc N(A; G)$. 
\end{lemma}
\begin{proof}[Proof of Lemma \ref{lemma:N_bar_upper_bound}]
Proceeding by way of contradiction, suppose $C_1\in (0, \overline{\mc N}(A; G))$ and $C_2\in (0, \infty)$ are constants for which \eqref{eq:sobolev_with_lower_order} holds for all $\bm u\in \mc W$. By definition of $\overline{\mc N}(A; G)$ there there is $x_0\in \overline\Omega$ for which 
\begin{equation}
\label{eq:C1_small}
	C_1< \mc S(A(x_0); G).
\end{equation} 
We fix any such $x_0$ and, for ease of notation, we set $M = A(x_0)$. By hypotheses on $A$, for all $\epsilon>0$ there is $\delta>0$ such that
\begin{equation*}
	(1 - \epsilon)\lb M\xi, \xi\rb
	\leq \lb A(x)\xi,\xi\rb
	\leq (1 + \epsilon)\lb M\xi, \xi\rb
	\quad \text{ for all }(x, \xi)\in (\overline\Omega \cap B(x_0, \delta))\times \bb R^n. 
\end{equation*}
Fix $\epsilon>0$ and choose $\delta\in (0, \epsilon^{1/p})$ as such. If $\bm u\in \mc W(x_0; \delta)$ (as defined in \eqref{eq:W_x0_delta}) then we have both
\begin{equation*}
\begin{split}
	\|\bm u\|_{\mc W_A}^p
	& = \sum_{j = 1}^d\int_{\Omega \cap B(x_0, \delta)}\lb A(x)\Grad u_j, \Grad u_j\rb^{p/2}\; \d x\\
	& \leq  (1 + \epsilon)^{p/2}\|\bm u\|_{\mc W_M}^p 
\end{split}
\end{equation*}
and, from the H\"older and Sobolev inequalties,  
\begin{equation*}
\begin{split}
	\|\bm u\|_{L^p(\Omega; \bb R^d)}^p
	& \leq C\delta^p\sum_{j = 1}^d\left(\int_{\Omega\cap B(x_0, \delta)}|u_j|^{p^*}\; \d x\right)^{p/p^*}\\
	& \leq C\epsilon\sum_{j = 1}^d\|u_j\|_{p^*}^p\\
	& \leq C_3\epsilon\|\bm u\|_{\mc W_M}^p, 
\end{split}
\end{equation*}
where $C_3$ is a positive constant depending only on $n$, $p$, and $\det M$ (in particular $C_3$ is independent of $\epsilon$ and $\bm u$). Therefore, for any $\epsilon>0$ and any $\bm u\in \mc W(x_0; \delta)$ that satisfies $\int_\Omega G(\bm u)\; \d x = 1$ we have
\begin{equation}
\label{eq:snarky}
\begin{split}
	1
	& \leq C_1\|\bm u\|_{\mc W_A}^p + C_2\|\bm u\|_{L^p(\Omega; \bb R^d)}^p\\
	& \leq (C_1(1 + \epsilon)^{p/2} + C_2C_3\epsilon)\|\bm u\|_{\mc W_M}^p. 
\end{split}
\end{equation}
Since $C_1< \mc S(M; G)$, we choose $\epsilon>0$ sufficiently small so that 
\begin{equation*}
	2(C_1(1 + \epsilon)^{p/2} + C_2C_3\epsilon)
	< \mc S(M; G) + C_1
\end{equation*}
and then obtain $\delta = \delta(\epsilon)\in (0, \epsilon^{1/p})$ for which \eqref{eq:snarky} holds whenever $\bm u\in \mc W(x_0, \delta)$ satisfies $\int_\Omega G(\bm u)\; \d x= 1$. For any such $\epsilon$, $\delta$ and $\bm u$ we obtain 
\begin{equation}
\label{eq:S(M,G)_virtual_bound}
	1
	\leq (C_1(1 + \epsilon)^{p/2} + C_2C_3\epsilon)\|\bm u\|_{\mc W_M}^p
	< \frac{\mc S(M; G) +C_1}{2}\|\bm u\|_{\mc W_M}^p. 
\end{equation}
From the definition of $\mc S(M; G)$ and from Remark \ref{remark:scaling_argument}, estimate \eqref{eq:S(M,G)_virtual_bound} implies $C_1\geq \mc S(M; G)$, and this inequality contradicts inequality \eqref{eq:C1_small}. 
\end{proof}
\begin{lemma}[$\epsilon$-sharp inequality]
\label{lemma:Nbar_epsilon_sharp_inequality}
Under the hypotheses of Proposition \ref{prop:N(A,G)_value}, for every $\epsilon>0$ there is $C_\epsilon>0$ such that the inequality 
\begin{equation}
\label{eq:epsilon_sharp_inequality}
	\left(\int_\Omega G(\bm u)\; \d x\right)^{p/p^*}
	\leq (\overline{\mc N}(A; G) + \epsilon)\|\bm u\|_{\mc W_A}^p + C_\epsilon\|\bm u\|_{L^p(\Omega; \bb R^d)}^p
\end{equation}
holds for all $\bm u\in \mc W$. In particular, $\mc N(A; G)\leq \overline{\mc N}(A; G)$. 
\end{lemma}
The proof of Lemma \ref{lemma:Nbar_epsilon_sharp_inequality} is similar in spirit to the proof of Theorem 9 of \cite{Aubin1976}, an english translation of which is provided in Theorm 4.5 of \cite{Hebey2000}. For the convenience of the reader, we provide the details. 
\begin{proof}[Proof of Lemma \ref{lemma:Nbar_epsilon_sharp_inequality}]
The proof is a partition of unity argument. Since $A$ satisfies \ref{item:A_continuous}, \ref{item:A_symmetric} and \ref{item:A_positive_definite}, and since $\overline \Omega$ is compact, for every $\epsilon_0>0$, there are $\delta>0$, $N\in \bb N$ and $\{x^i\}_{i = 1}^N\subset \overline\Omega$ for which  both $\overline \Omega \subset \bigcup_{i = 1}^NB(x^i, \delta)$ and 
\begin{equation}
\label{eq:A_partition_estimates}
	(1 - \epsilon_0)\lb A(x^i)\xi, \xi\rb
	\leq \lb A(x)\xi, \xi\rb
	\leq (1 + \epsilon_0)\lb A(x^i)\xi, \xi\rb
	\quad \text{ for all }(x, \xi)\in (\overline\Omega \cap B_i)\times \bb R^n, 
\end{equation}
where, for ease of notation, we set $B_i = B(x^i, \delta)$. Fix $\epsilon_0>0$ and let $\delta>0$, $N\in \bb N$ and $\{x^i\}_{i = 1}^N$ be as such. Let $\{\zeta_i\}_{i = 1}^N$ be a smooth partition of unity subordinate to $\{B_i\}_{i = 1}^N$ and define
\begin{equation*}
	\eta_i = \frac{\zeta_i^{\lfloor p\rfloor  + 1}}{\sum_{j = 1}^N \zeta_j^{\lfloor p\rfloor + 1}}, 
\end{equation*}
where $\lfloor p\rfloor$ is the greatest integer that does not exceed $p$. Evidently $\supp \eta_i \subset B_i$, $\eta_i^{1/p}\in C^1$ and $\sum_{i = 1}^N \eta_i \equiv 1$ on $\overline\Omega$. For any $\bm u\in \mc W$ we have
\begin{equation}
\label{eq:epsilon_sharp_start}
\begin{split}
	\left(\int_\Omega G(\bm u)\; \d x\right)^{p/p^*}
	& = \left(\int_\Omega\left(\sum_{i = 1}^N \eta_i G(\bm u)^{p/p^*}\right)^{p^*/p}\; \d x\right)^{p/p^*}\\
	& = \left(\int_\Omega\left(\sum_{i = 1}^N \eta_i |\bm u|_p^p G\left(\frac{\bm u}{|\bm u|_p}\right)^{p/p^*}\right)^{p^*/p}\; \d x\right)^{p/p^*}\\
	& \leq \sum_{i = 1}^N \left(\int_\Omega \left(\eta_i^{1/p}|\bm u|_p\right)^{p^*}G\left(\frac{\bm u}{|\bm u|_p}\right)\; \d x\right)^{p/p^*}\\
	& = \sum_{i = 1}^N\left(\int_\Omega G(\eta_i^{1/p}\bm u)\; \d x\right)^{p/p^*}\\
	& \leq \sum_{i = 1}^N\mc S(A(x^i); G)\|\eta_i^{1/p}\bm u\|_{\mc W_{A(x^i)}}^p\\
	& \leq \overline{\mc N}(A; G)\sum_{i = 1}^N\|\eta_i^{1/p}\bm u\|_{\mc W_{A(x^i)}}^p. 
\end{split}
\end{equation}
We proceed to estimate the sum on the right-most side of this string of inequalities. For each $(i, j)\in \{1, \ldots, N\}\times\{1, \ldots, d\}$, using Cauchy's inequality and the fact that \eqref{eq:A_partition_estimates} holds on $\supp(\eta_i)$ we obtain 
\begin{equation*}
	0
	\leq \lb A(x^i)\Grad(\eta_i^{1/p}u_j), \Grad(\eta_i^{1/p}u_j)\rb
	\leq f_{ij}^2 + h_{ij}^2, 
\end{equation*}
where $f_{ij}$ and $h_{ij}$ are defined by  
\begin{equation}
\label{eq:fij_hij_definitions}
\begin{split}
	f_{ij}^2 & = \eta_i^{2/p}\left(\frac 1{1 - \epsilon_0}\lb A(x)\Grad u_j, \Grad u_j\rb + \epsilon^2 \Lambda^2|\Grad u_j|^2\right)\\
	h_{ij}^2 & = u_j^2\left(\frac 1{\epsilon^2} + \Lambda\right)\|\Grad(\eta_i^{1/p})\|_\infty^2 
\end{split}
\end{equation}
and $\Lambda$ is as in \eqref{eq:A_upper_bounded}. According to this estimate and Lemma \ref{lemma:elementary_Lp}, for all $\epsilon\in (0, 1)$ we have
\begin{equation}
\label{eq:each_ij_estimate}
\begin{split}
	\sum_{i = 1}^N\|\eta_i^{1/p}\bm u\|_{\mc W_{A(x^i)}}^p
	& \leq \sum_{i = 1}^N\sum_{j= 1}^d\int_\Omega \left(f_{ij}^2 + h_{ij}^2\right)^{p/2}\; \d x\\
	& \leq \sum_{i = 1}^N\sum_{j= 1}^d\int_\Omega \left(f_{ij} + h_{ij}\right)^p\; \d x\\
	& \leq \sum_{i = 1}^N\sum_{j= 1}^d\left((1 + \epsilon)\|f_{ij}\|_p^p + \frac{C(p)}{\epsilon^p}\|h_{ij}\|_p^p\right). 
\end{split}
\end{equation}
We proceed to separately estimate the sum over $i$ and $j$ of the norms $\|f_{ij}\|_p^p$ and $\|h_{ij}\|_p^p$. From the expression of $h_{ij}$ in \eqref{eq:fij_hij_definitions} we easily estimate
\begin{equation}
\label{eq:sum_hij_estimate}
\begin{split}
	\sum_{i = 1}^N\sum_{j = 1}^d\|h_{ij}\|_p^p
	& \leq N\max_i\|\Grad(\eta_i^{1/p})\|_\infty^p\left(\frac 1{\epsilon^2} + \Lambda\right)^{p/2}\|\bm u\|_{L^p(\Omega; \bb R^d)}^p. 
\end{split}
\end{equation}
To estimate the $p$-norm of $f_{ij}$ we define $\phi_j$ and $\psi_j$ by 
\begin{equation*}
\begin{split}
	\phi_j^2 & = \frac 1{1 - \epsilon_0}\lb A(x)\Grad u_j, \Grad u_j\rb\\
	\psi_j^2 & = \epsilon^2 \Lambda^2|\Grad u_j|^2
\end{split}	
\end{equation*}
so that the expression of $f_{ij}^2$ in \eqref{eq:fij_hij_definitions} reads $f_{ij}^2 = \eta_i^{2/p}(\phi_j^2 + \psi_j^2)$. For each $j \in \{1, \ldots, d\}$, using the fact that $\sum_{i = 1}^N\eta_i\equiv 1$ together with Lemma \ref{lemma:elementary_Lp} (applied with $\sqrt \epsilon$ in place of $\epsilon$) we have
\begin{equation*}
\begin{split}
	\sum_{i = 1}^N\|f_{ij}\|_p^p
	& = \int_\Omega\left(\sum_{i = 1}^N\eta_i\right)(\phi_j^2 + \psi_j^2)^{p/2}\; \d x\\
	& \leq \int_\Omega (\phi_j + \psi_j)^p\; \d x\\
	& \leq (1 + \epsilon^{1/2})\|\phi_j\|_p^p + \frac{C(p)}{\epsilon^{p/2}}\|\psi_j\|_p^p\\
	& \leq \left(\frac{1 + \epsilon^{1/2}}{(1 - \epsilon_0)^{p/2}} + C(p)\Lambda^p\frac{\epsilon^{p/2}}{\tau^{p/2}}\right)\|u_j\|_{W_{0, A}^{1, p}}^p, 
\end{split}
\end{equation*}
where $\tau$ is as in assumption \ref{item:A_positive_definite}. Summing this estimate over $j$ we obtain 
\begin{equation}
\label{eq:sum_fij_estimate}
\begin{split}
	\sum_{i = 1}^N\sum_{j = 1}^d\|f_{ij}\|_p^p
	& \leq \left(\frac{1 + \epsilon^{1/2}}{(1 - \epsilon_0)^{p/2}} + C(p)\Lambda^p\frac{\epsilon^{p/2}}{\tau^{p/2}}\right)\|\bm u\|_{\mc W_A}^p. 
\end{split}
\end{equation}	
Using estimates \eqref{eq:sum_fij_estimate} and \eqref{eq:sum_hij_estimate} in \eqref{eq:each_ij_estimate} gives
\begin{equation}
\label{eq:epsilon_sharp_tedious_estimate}
\begin{split}
	\sum_{i = 1}^N\|\eta_i^{1/p}\bm u\|_{\mc W_{A(x^i)}}^p
	& \leq (1 + \epsilon)\left(\frac{1 + \epsilon^{1/2}}{(1 - \epsilon_0)^{p/2}} + C(p)\Lambda^p\frac{\epsilon^{p/2}}{\tau^{p/2}}\right)\|\bm u\|_{\mc W_A}^p\\
	& + \frac{C(p)}{\epsilon^p}N\max_i\|\Grad(\eta_i^{1/p})\|_\infty^p\left(\frac 1{\epsilon^2} + \Lambda\right)^{p/2}\|\bm u\|_{L^p(\Omega; \bb R^d)}^p.
\end{split}
\end{equation}
Since both $\epsilon\in (0, 1)$ and $\epsilon_0\in (0, 1)$ are arbitrary, using estimate \eqref{eq:epsilon_sharp_tedious_estimate} in estimate \eqref{eq:epsilon_sharp_start} establishes inequality \eqref{eq:epsilon_sharp_inequality}. 
\end{proof}
In view of the equality $\overline{\mc N}(A; G) = \mc N(A;G)$ and Lemma \ref{lemma:Nbar_epsilon_sharp_inequality} we obtain the following corollary. 
\begin{coro}
\label{coro:epsilon_sharp_inequality}
Under the hypotheses of Proposition \ref{prop:N(A,G)_value}, for every $\epsilon>0$ there is $C_\epsilon>0$ such that the inequality 
\begin{equation*}
	\left(\int_\Omega G(\bm u)\; \d x\right)^{p/p^*}
	\leq (\mc N(A; G) + \epsilon)\|\bm u\|_{\mc W_A}^p + C_\epsilon\|\bm u\|_{L^p(\Omega; \bb R^d)}^p
\end{equation*}
holds for all $\bm u\in \mc W$. 
\end{coro}
\section{Sufficient Condition for Existence of a Minimizer}
\label{s:sufficient_condition_for_minimizer}
In this section we establish a sufficient condition for the existence of a minimizer of the functional $\Phi_{A, F}$ defined in \eqref{eq:intro_Phi_AF} subject to the constraint $\mc M$ defined in \eqref{eq:the_constraint}. Given $G$ satisfying \ref{item:G_positively_homogeneous}, it will be convenient to utilize the notation 
\begin{equation*}
	\Psi(\bm u) = \Psi_G(\bm u) = \int_\Omega G(\bm u)\; \d x. 
\end{equation*}
We define 
\begin{equation}
\label{eq:energy_quotient}
	Q_{A, F, G}(\bm u)
	= \frac{\Phi_{A, F}(\bm u)}{\Psi(\bm u)^{p/p^*}}
	\qquad \text{ for }\bm u\in \mc W\setminus \{0\} 
\end{equation}
and we observe that when $A$, $F$ and $G$ satisfy \eqref{assumptions:AFG}, Lemma \ref{lemma:Phi_AF_coercive} guarantees that $Q_{A, F, G}$ is bounded below by a positive constant. For such $A$, $F$ and $G$, the quantity
\begin{equation*}
\begin{split}
	K(A, F, G)^{-1} 
	& = \inf\{\Phi_{A, F}(\bm u): \bm u\in \mc M\}\\
	& = \inf\{Q_{A, F, G}(\bm u): \bm u\in \mc W\setminus\{0\}\}, 
\end{split}
\end{equation*}
is well-defined and positive. For ease of notation, in what follows we do not indicate the dependence on $A$, $F$ or $G$ in the notations for $\Phi$, $Q$ or $K$. For example, the notation $\Phi$ is understood to stand for $\Phi_{A, F}$. The following proposition is the main result of this section. 
\begin{prop}[Sufficient condition for minimizer]
\label{prop:sufficient_condition_for_minimizer}
Let $p\in (1, n)$, let $\Omega\subset \bb R^n$ be a bounded open set and suppose $A$, $F$ and $G$ satisfy assumptions \eqref{assumptions:AFG}. If 
\begin{equation}
\label{eq:energy_threshhold}
	\mc N(A; G)< K(A, F, G)
\end{equation} 
then there is $\bm u\in \mc M$ for which $\Phi(\bm u) = K(A, F, G)^{-1}$. 
\end{prop}
The remainder of this section will be devoted to the proof of Proposition \ref{prop:sufficient_condition_for_minimizer}. This will be accomplished with the aid of a series of lemmas. The first such lemma guarantees the existence of a convenient minimizing sequence for $\Phi|_{\mc M}$. It follows from a standard application of Ekeland's variational principle. For the convenience of the reader, a proof is provided in the appendix. 
\begin{lemma}
\label{lemma:good_minimizing_sequence}
Let $p\in (1, n)$ and let $\Omega\subset \bb R^n$ be a bounded open set. If $A$, $F$ and $G$ satisfy assumptions \eqref{assumptions:AFG} then there is a minimizing sequence $(\bm u^k)_{k = 1}^\infty\subset \mc M$ for $K^{-1}$ for which 
\begin{equation}
\label{eq:some_derivatives_vanishing}
	\Phi'(\bm u^k)- \frac{p}{p^*}K^{-1}\Psi'(\bm u^k) \to 0
	\qquad \text{ in }\mc W_A'. 
\end{equation}
\end{lemma}
\begin{lemma}
\label{lemma:simon_type_inequality}
Let $p\in (1, \infty)$ and let $\Omega \subset \bb R^n$ be a bounded open set. If $A$ satisfies \ref{item:A_continuous}, \ref{item:A_symmetric} and \ref{item:A_positive_definite} then there is $C>0$ such that for all $(x, \xi, \zeta)\in \overline\Omega\times \bb R^n\times \bb R^n$,
\begin{equation*}
	D_{A(x)}(\xi, \zeta)
	\geq C\begin{cases}
	|\xi - \zeta|^p & \text{ if }p\geq 2\\
	|\xi - \zeta|^2(|\xi| + |\zeta|)^{p - 2} & \text{ if }1< p \leq 2,  
	\end{cases}
\end{equation*}
where 
\begin{equation}
\label{eq:p_difference}
	D_{A(x)}(\xi, \zeta)
	= \left(\lb A(x)\xi, \xi\rb^{\frac{p - 2}2}A(x)\xi - \lb A(x)\zeta, \zeta\rb^{\frac{p - 2}{2}}A(x)\zeta\right)\cdot (\xi - \zeta).
\end{equation}
\end{lemma}
\begin{proof}
The result is known in the case $A\equiv I$, see for example Theorem 1 and Theorem 2 of \cite{Cheng1998}. We proceed to reduce the case $A\not\equiv I$ to the case $A \equiv I$. By \ref{item:A_symmetric} and \ref{item:A_positive_definite}, for each $x\in \overline\Omega$ there is an orthogonal matrix $O(x)$ such that
\begin{equation*}
	O(x)A(x)O(x)^\top = \diag(a_1(x), \ldots, a_n(x))=:\delta(x), 
\end{equation*}
where $0< \tau\leq a_1(x)\leq \ldots \leq a_n(x)$ are the eigenvalues of $A(x)$. Setting $B(x) = O(x)^\top\delta(x)^{1/2}O(x)$ we have both $B(x)^\top = B(x)$ and $A(x) = B(x)^\top B(x)$ for all $x$. Therefore, 
\begin{equation}
\label{eq:D_A(xi,zeta)_initial_estimate}
\begin{split}
	\lefteqn{D_{A(x)}(\xi, \zeta)}\\
	& = D_I(B(x)\xi, B(x)\zeta)\\
	& \geq C\begin{cases}
	|B(x)\xi - B(x)\zeta|^p & \text{ if }p\geq 2\\
	|B(x)\xi - B(x)\zeta|^2(|B(x)\xi| + |B(x)\zeta|)^{p - 2} & \text{ if }1< p\leq 2, 
	\end{cases}
\end{split}
\end{equation}
where the final estimate follows from the fact that the Lemma is known in the case $A\equiv I$. For any $(x, \xi)\in \overline \Omega \times \bb R^n$ we have $|B(x)\xi|^2= \lb A(x)\xi,\xi\rb$ and thus $\tau|\xi|^2 \leq |B(x)\xi|^2 \leq \Lambda|\xi|^2$, where $\Lambda$ is as in \eqref{eq:A_upper_bounded}. Using this in \eqref{eq:D_A(xi,zeta)_initial_estimate} gives the asserted inequality. 
\end{proof}
The following lemma is modified from Theorem 1.1. of \cite{deValeriolaWillem2009}. To state the lemma we define $T:\bb R\to \bb R$ by 
\begin{equation}
\label{eq:T}
	T(r)
	= \begin{cases}
	r & \text{ if }|r|\leq 1\\
	\frac r{|r|} & \text{ if }|r| >1. 
	\end{cases}	
\end{equation}
\begin{lemma}
\label{lemma:from_deValeriola_Willem}
Let $p>1$, let $\Omega\subset \bb R^n$ be bounded and open and suppose $A$ satisfies \ref{item:A_continuous}, \ref{item:A_symmetric} and \ref{item:A_positive_definite}. If $(v^k)_{k = 1}^\infty \subset W_0^{1,p }(\Omega)$ with $v^k\weakconv v$ weakly in $W_0^{1, p}(\Omega)$ and 
\begin{equation}
\label{eq:de_valeriola_willem_convergence_assumption}
	\int_\Omega \left(
	\lb A(x)\Grad v^k, \Grad v^k\rb^{\frac{p - 2}2}A(x)\Grad v^k - \lb A(x)\Grad v,  \Grad v\rb^{\frac{p - 2}2}A(x)\Grad v\right)\cdot \Grad (T\circ (v^k- v))\; \d x
	\to 0
\end{equation}
then there is a subsequence along which $|\Grad v^k -\Grad v|\to 0$ a.e. in $\Omega$.  
\end{lemma}
\begin{proof}
Rellich's Theorem guarantees the existence of a subsequence of $(v^k)_{k = 1}^\infty$ along which $v^k\to v$ a.e. in $\Omega$. Defining 
\begin{equation*}
	E_k = \{x\in \Omega : |v^k - v|\leq 1\}
\end{equation*}
and
\begin{equation*}
	e_k(x) 
	= D_{A(x)}(\Grad v^k, \Grad v), 
\end{equation*}	
where $D_{A(x)}(\xi, \zeta)$ is as in \eqref{eq:p_difference}, assumption \eqref{eq:de_valeriola_willem_convergence_assumption} gives $\int_{E_k}e_k\; \d x\to 0$. Moreover, Lemma \ref{lemma:simon_type_inequality} guarantees that $e_k\chi_{E_k}\geq 0$, so there is a subsequence along which $e_k\chi_{E_k}\to 0$ a.e. in $\Omega$. Since, in addition, $v^k\to v$ a.e. in $\Omega$ we have $\chi_{E_k}\to 1$ a.e. in $\Omega$. We therefore deduce that $e_k\to 0$ a.e. in $\Omega$. Applying Lemma \ref{lemma:simon_type_inequality} we find that along such a subsequence, $|\Grad v^k - \Grad v|\to 0$ a.e. in $\Omega$. 
\end{proof}
\begin{proof}[Proof of Proposition \ref{prop:sufficient_condition_for_minimizer}]
Lemma \ref{lemma:good_minimizing_sequence} ensures the existence of a sequence $(\bm u^k)_{k = 1}^\infty\subset \mc M$ for which both $\Phi(\bm u^k)\to K^{-1}$ and  \eqref{eq:some_derivatives_vanishing} holds, so we fix such a sequence. Lemma \ref{lemma:Phi_AF_coercive} ensures that $\Phi$ is coercive, so $(\bm u^k)_{k = 1}^\infty$ is bounded in $\mc W$. The reflexivity of $\mc W$ and the compactness of the embedding $\mc W\hookrightarrow L^p(\Omega; \bb R^d)$ guarantees the existence of $\bm u\in \mc W$ and a subsequence of $(\bm u^k)_{k = 1}^\infty$ along which 
\begin{equation*}
\begin{split}
	\bm u^k\weakconv \bm u& \text{ weakly in }\mc W\\
	\bm u^k\to \bm u & \text{ in }L^p(\Omega; \bb R^d)\\
	\bm u^k\to \bm u & \text{ a.e. in }\Omega. 
\end{split}
\end{equation*}
The continuity of $F$ and the convergence $\bm u^k\to \bm u$ a.e. in $\Omega$ ensures that $F(\bm u^k)\to F(\bm u)$ a.e. in $\Omega$. Moreover, Lemma \ref{lemma:properties_homogeneous_functions} \ref{item:homogeneous_upper_bound} gives $|F(\bm u^k)|\leq M_F |\bm u^k|_p^p$, so from the convergence $\bm u^k\to \bm u$ in $L^p(\Omega; \bb R^d)$ and the generalized Dominated Convergence Theorem we obtain $\int_\Omega F(\bm u^k)\; \d x\to \int_\Omega F(\bm u)\; \d x$. Now, from the definition of $K$, the assumption $K> \mc N$ ensures the existence of $C>0$ such that the estimate
\begin{equation*}
	1
	= \left(\int_\Omega G(\bm u^k)\; \d x\right)^{p/p^*}
	\leq \frac{K + \mc N}{2}\|\bm u^k\|_{\mc W_A}^p + C\|\bm u^k\|_{L^p(\Omega; \bb R^d)}^p
\end{equation*}
holds for all $k$. For any such $C$ we have
\begin{equation*}
\begin{split}
	K^{-1} + \circ(1)
	& = \Phi(\bm u^k)\\
	& = \|\bm u^k\|_{\mc W_A}^p - \int_\Omega F(\bm u)\; \d x + \circ(1)\\
	& \geq \frac 2{K + \mc N}\left(1  - C\|\bm u^k\|_{L^p(\Omega; \bb R^d)}^p\right) - \int_\Omega F(\bm u)\; \d x + \circ(1). 
\end{split}
\end{equation*}
Letting $k\to \infty$ in this estimate gives
\begin{equation*}
\begin{split}
	0
	& < \frac{2}{K + \mc N} - \frac 1 K\\
	& \leq \frac{2C}{K + \mc N}\|\bm u\|_p^p + \int_\Omega F(\bm u)\; \d x\\
	&\leq \frac{2C}{K + \mc N}\|\bm u\|_p^p + M_F\|\bm u\|_p^p, 
\end{split}
\end{equation*}
so $\bm u\not\equiv 0$. We proceed to show that $\bm u\in \mc M$. From the continuity of $G$, the a.e. convergence $\bm u^k \to \bm u$ and the assumption $(\bm u^k)_{k = 1}^\infty\subset \mc M$, Fatou's lemma gives
\begin{equation}
\label{eq:applied_fatous}
	\int_\Omega G(\bm u)\; \d x 
	\leq \liminf_k\int_\Omega G(\bm u^k)\; \d x
	= 1, 
\end{equation}
so we only need to show that $\int_\Omega G(\bm u)\; \d x\geq 1$. Since \eqref{eq:some_derivatives_vanishing} holds, for any $\bm \varphi \in \mc W$ we have
\begin{equation*}
\begin{split}
	\circ(1)\|\bm \varphi\|_{\mc W_A}
	& = \lb \Phi'(\bm u^k), \bm \varphi\rb - \frac{p}{p^*}K^{-1}\lb \Psi'(\bm u^k), \bm \varphi\rb\\
	& = p\sum_{j = 1}^d\int_\Omega \lb A(x)\Grad u_j^k, \Grad u_j^k\rb^{\frac{p - 2}{2}}\lb A(x)\Grad u_j^k, \Grad \varphi_j\rb\; \d x\\
	& - \int_\Omega \lb\Grad F(\bm u^k), \bm \varphi\rb - \frac{p}{p^*K}\int_\Omega \lb \Grad G(\bm u^k), \bm \varphi\rb\; \d x.
\end{split}
\end{equation*}
Using this equality with 
\begin{equation}
\label{eq:phi1_test_function}
	\bm \varphi = (T\circ(u_1^k - u_1), 0, \ldots, 0), 
\end{equation}
where $T$ is defined in \eqref{eq:T}, and using the notation 
\begin{equation*}
	h(u) = \lb A(x)\Grad u, \Grad u\rb^{(p -2)/2}A(x)\Grad u
\end{equation*}
for $u\in W_0^{1, p}(\Omega)$, we obtain
\begin{equation}
\label{eq:drink_espresso}
\begin{split}
	p\int_\Omega
	& (h(u_1^k) - h(u_1))\cdot \Grad (T\circ (u_1^k - u_1))\; \d x\\
	& = \circ(1)\|T\circ(u_1^k - u_1)\|_{W_{0, A}^{1, p}} + \int_\Omega \partial_1F(\bm u^k)T(u_1^k - u_1)\; \d x\\
	& + \frac{p}{p^*K}\int_\Omega \partial_1G(\bm u^k)T(u_1^k - u_1)\; \d x\\
	& - p\int_\Omega h(u_1)\cdot \Grad (T\circ(u_1^k - u_1))\; \d x.  
\end{split}
\end{equation}
For any $r\in [1, \infty)$, a routine argument involving the Dominated Convergence Theorem shows that $\|T(u_1^k - u_1)\|_{L^r(\Omega)}\to 0$. Therefore, using the degree-$(p- 1)$ homogenity of $\partial_1F$ and the fact that $(\bm u^k)_{k = 1}^\infty$ is bounded in $L^p(\Omega; \bb R^d)$ we obtain 
\begin{equation}
\label{eq:partial1_F_term}
\begin{split}
	\abs{\int_\Omega \partial_1F(\bm u^k)T(u_1^k - u_1)\; \d x}
	& \leq \int_\Omega |\partial_1F(\bm u^k)||T(u_1^k - u_1)|\; \d x\\
	& \leq \max_{x\in \bb S_p^{d - 1}}|\partial_1 F(s)|\int_\Omega |\bm u^k|_p^{p-1}|T(u_1^k - u_1)|\; \d x\\
	& \leq \max_{x\in \bb S_p^{d - 1}}|\partial_1 F(s)|\|\bm u^k\|_{L^p(\Omega; \bb R^d)}^{p - 1}\|T(u_1^k - u_1)\|_{L^p(\Omega)}\\
	& \to 0. 
\end{split}
\end{equation}
Similarly, using the degree-$(p^* - 1)$ homogenity of $\partial_1G$ and the fact that $T(u_1^k - u_1)\to 0$ in $L^{p^*}$ we obtain 
\begin{equation}
\label{eq:partial1_G_term}
\begin{split}
	\abs{\int_\Omega\partial_1 G(\bm u^k)T(u_1^k - u_1)\; \d x}
	& \to 0. 
\end{split}
\end{equation}
From the $\mc W_A$-weak convergence $\bm u^k \weakconv \bm u$ and since $T'(u_1^k - u_1) = 0$ on $\{|u_1^k - u_1|\geq 1\}$ we have
\begin{equation}
\label{eq:cruz_baby}
\begin{split}
	\circ(1)
	& = \int_\Omega\lb A(x)\Grad u_1, \Grad u_1\rb^{(p - 2)/2}\lb A(x)\Grad u_1, \Grad (T\circ(u_1^k - u_1))\rb\; \d x\\
	& + \int_{\{|u_1^k - u_1|\geq 1\}}\lb A(x)\Grad u_1, \Grad u_1\rb^{(p - 2)/2}\lb A(x)\Grad u_1, \Grad(u_1^k - u_1)\rb\; \d x.  
\end{split}
\end{equation}
Moreover, since $|\Grad u_1^k|$ is bounded in $L^p(\Omega)$ and since the $L^p$-convergence $u_1^k\to u_1$ guarantees that $|\{|u_1^k - u_1|\geq 1\}|\to 0$ , the last term in \eqref{eq:cruz_baby} satisfies
\begin{equation*}
\begin{split}
	\lefteqn{\abs{\int_{\{|u_1^k - u_1|\geq 1\}}\lb A(x)\Grad u_1, \Grad u_1\rb^{(p - 2)/2}\lb A(x)\Grad u_1, \Grad(u_1^k - u_1)\rb\; \d x}}\\
	& \leq C\int_{\{|u_1^k - u_1|\geq 1\}}|\Grad u_1|^{p - 1}|\Grad(u_1^k - u_1)|\; \d x\\
	& \leq C\|\Grad u_1\|_{L^p(\{|u_1^k - u_1|\geq 1\})}^{p - 1}\|\Grad(u_1^k - u_1)\|_{L^p(\Omega)}\\
	& \to 0 
\end{split}
\end{equation*}
for some $C = C(\tau, \Lambda)>0$. In particular, \eqref{eq:cruz_baby} implies that 
\begin{equation}
\label{eq:weak_conv_term}
	\int_\Omega \lb A(x)\Grad u_1, \Grad u_1\rb^{(p - 2)/2}\lb A(x)\Grad u_1, \Grad(T\circ(u_1^k - u_1))\rb\; \d x
	= \circ(1). 
\end{equation}
Using \eqref{eq:partial1_F_term}, \eqref{eq:partial1_G_term} and \eqref{eq:weak_conv_term} together with the fact that $(\|T\circ(u_1^k - u_1)\|_{W_{0, A}^{1, p}})_{k = 1}^\infty$ is bounded, \eqref{eq:drink_espresso} yields
\begin{equation*}
	\int_\Omega (h(u_1^k) - h(u_1))\cdot \Grad (T\circ(u_1^k - u_1))\; \d x
	= \circ(1). 
\end{equation*}
From this estimate and from Lemma \ref{lemma:from_deValeriola_Willem} we find that there is a subsequence along which $\Grad u_1^k\to \Grad u_1$ a.e. in $\Omega$. For each $j= 2, \ldots, d$, the same argument, but using $\bm \varphi = (0, \ldots, 0, T\circ(u_j^k - u_j), 0, \ldots, 0)$ in place of the test function in \eqref{eq:phi1_test_function} guarantees the existence of a subsequence of $\bm u^k$ along which $\Grad u_j^k\to \Grad u_j$ a.e. in $\Omega$. We now select a subsequence such that this holds simultaneously for all $j$. The Brezis-Lieb lemma \cite{BrezisLieb1983} gives
\begin{equation*}
\begin{split}
	\int_\Omega & \lb A(x)\Grad u_j^k, \Grad u_j^k\rb^{p/2}\; \d x\\
	& = \int_\Omega \lb A(x)\Grad u_j, \Grad u_j\rb^{p/2}\; \d x
	+ \int_\Omega \lb A(x)\Grad (u_j^k - u_j), \Grad (u_j^k - u_j)\rb^{p/2}\; \d x
	+ \circ(1)
\end{split}
\end{equation*}
for all $j \in \{1, \ldots, d\}$. Summing this equality over $j$ gives
\begin{equation*}
	\|\bm u^k\|_{\mc W_A}^p 
	= \|\bm u\|_{\mc W_A}^p + \|\bm u^k - \bm u\|_{\mc W_A}^p + \circ(1)
\end{equation*}
and therefore, by the assumption $\Phi(\bm u^k)\to K^{-1}$, we have
\begin{equation}
\label{eq:will_use_later}
\begin{split}
	K^{-1} + \circ(1)
	& = \Phi(\bm u^k)\\
	& = \|\bm u^k\|_{\mc W_A}^p - \int_\Omega F(\bm u) + \circ(1)\\
	& = \Phi(\bm u) + \|\bm u^k - \bm u\|_{\mc W_A}^p + \circ(1)\\
	& \geq K^{-1}\left(\int_\Omega G(\bm u)\; \d x\right)^{p/p^*} + \|\bm u^k - \bm u\|_{\mc W_A}^p + \circ(1).
\end{split}
\end{equation}
For any $\epsilon>0$, applying the $\epsilon$-sharp inequality in Corollary \ref{coro:epsilon_sharp_inequality} to the second term on the right-most side of \eqref{eq:will_use_later} we obtain 
\begin{equation*}
\begin{split}
	K^{-1} &  + \circ(1) \\
	& \geq K^{-1}\left(\int_\Omega G(\bm u)\; \d x\right)^{p/p^*} + (\mc N + \epsilon)^{-1}\left(\int_\Omega G(\bm u^k - \bm u)\; \d x\right)^{p/p^*}\\
	&  - C_\epsilon(\mc N + \epsilon)^{-1}\|\bm u^k- \bm u\|_{L^p(\Omega; \bb R^d)}^p + \circ(1)\\
	& = K^{-1}\left(\int_\Omega G(\bm u)\; \d x\right)^{p/p^*} + (\mc N + \epsilon)^{-1}\left(\int_\Omega G(\bm u^k - \bm u)\; \d x\right)^{p/p^*} +\circ(1). 
\end{split}
\end{equation*}
Rearranging this inequality and applying the Brezis-Lieb type lemma to $\int_\Omega G(\cdot)\; \d x$, see Lemma 2.1 of \cite{Amster2002}, gives
\begin{equation*}
\begin{split}
	K^{-1}\left(1 - \left(\int_\Omega G(\bm u)\; \d x\right)^{p/p^*}\right)
	& \geq \frac{1}{\mc N + \epsilon}\left(\int_\Omega G(\bm u^k - \bm u)\; \d x\right)^{p/p^*} + \circ(1)\\
	& = \frac{1}{\mc N + \epsilon}\left(\int_\Omega G(\bm u^k)\; \d x - \int_\Omega G(\bm u)\; \d x\right)^{p/p^*} + \circ(1)\\
	& = \frac{1}{\mc N + \epsilon}\left(1 - \int_\Omega G(\bm u)\; \d x\right)^{p/p^*} + \circ(1)\\
	& \geq \frac{1}{\mc N + \epsilon}\left(1 - \left(\int_\Omega G(\bm u)\; \d x\right)^{p/p^*}\right) + \circ(1). 
\end{split}
\end{equation*}
Letting $k\to \infty$ we find that for all $\epsilon>0$, 
\begin{equation*}
	0\leq \left(1 - \left(\int_\Omega G(\bm u)\; \d x\right)^{p/p^*}\right)\left(K^{-1} - (\mc N + \epsilon)^{-1}\right). 
\end{equation*}
Choosing $\epsilon>0$ small enough so that $\mc N + \epsilon < K$ yields $\int_\Omega G(\bm u)\; \d x \geq 1$ which, when combined with estimate \eqref{eq:applied_fatous} gives $\bm u\in \mc M$. Using this in estimate \eqref{eq:will_use_later} we obtain $\|\bm u^k- \bm u\|_{\mc W_A}\to 0$, so the continuity of $\Phi$ gives $\Phi(\bm u) = K^{-1}$.
\end{proof}
Up to a positive constant multiple, any minimizer of $\Phi_{A, F}$ subject to the constraint $\bm u\in \mc M$ is a nontrivial weak solution to \eqref{eq:main_problem}, so we obtain the following corollary. 
\begin{coro}
\label{coro:minimizing_solution}
Under the hypotheses of Proposition \ref{prop:sufficient_condition_for_minimizer} there is a nontrivial solution to problem \eqref{eq:main_problem}. 
\end{coro}
\section{Proofs of Existence Theorems}
\label{s:proofs_of_existence}
In this section we give proofs of Theorems \ref{theorem:existence_gamma_small1}, \ref{theorem:existence_gamma_large} and \ref{theorem:boundary_minimizer_existence}. In each of the proofs we verify that the energy smallness condition in \eqref{eq:energy_threshhold} is satisfied by establishing the existence of $\bm u\in \mc M$ for which $\Phi(\bm u)< \mc N(A; G)^{-1}$. 
\subsection{Construction of Test Function for $\gamma\leq p$}
This subsection is devoted to the proof of Theorem \ref{theorem:existence_gamma_small1}. In the proof we will employ the Hardy-Sobolev inequality which we recall now. For $p\in (1, n)$, for a bounded open set $\Omega \subset \bb R^n$, for $\gamma\in (0, p]$ and for $x_0\in \Omega$, the Hardy-Sobolev inequality ensures the existence of an optimal positive constant $K_0(n, p, \gamma,\Omega, x_0)$ such that
\begin{equation}
\label{eq:hardy_sobolev}
	\int_\Omega |u|^p\; \d x 
	\leq K_0(n, p, \gamma, \Omega, x_0)^p\int_\Omega |x - x_0|^\gamma|\Grad u|^p\; \d x
	\qquad \text{ for all }u\in W_0^{1, p}(\Omega). 
\end{equation}
This inequality is a particular case of the Caffarelli-Kohn-Nirenberg inequalities \cite{CaffarelliKohnNirenberg1984}. When $\gamma = p$, the constant $K_0$ on the right-hand side of \eqref{eq:hardy_sobolev} is independent of $x_0$ and $\Omega$. More specifically, $K_0(n,p, p, \Omega, x_0) = p/n$, see \cite{HardyLittlewoodPolya1934} for the case $p = 2$ and \cite{AbdellaouiColoradoPeral2005} for the case $p\neq 2$. 

\begin{proof}[Proof of Theorem \ref{theorem:existence_gamma_small1}]
For simplicity we assume $x_0 = 0\in\Omega$, in which case assumption \eqref{eq:A_lower_expansion} becomes
\begin{equation}
\label{eq:A_lower_expansion_zero}
	\lb A(x)\xi, \xi\rb^{p/2}
	\geq \lb A(0)\xi, \xi\rb^{p/2} + C_0|x |^\gamma |\xi|^p
	\qquad\text{ for all }(x, \xi)\in \Omega\times \bb R^n.   
\end{equation}
Using this estimate together with the sharp Hardy-Sobolev inequality \eqref{eq:hardy_sobolev} we have for any $\bm u\in \mc W$, 
\begin{equation}
\label{eq:Theta_not_empty}
\begin{split}
	\|\bm u\|_{\mc W_A}^p
	& \geq \|\bm u\|_{\mc W_{A(0)}}^p + C_0\sum_{j = 1}^d\int_\Omega |x|^\gamma |\Grad u_j|^p\; \d x\\
	& \geq \|\bm u\|_{\mc W_{A(0)}}^p + \frac{C_0}{K_0^p}\|\bm u\|_{L^p(\Omega; \bb R^d)}^p\\
	& \geq \mc S(A(0); G)^{-1}\left(\int_\Omega G(\bm u)\; \d x\right)^{p/p^*} + \frac{C_0}{K_0^p}\|\bm u\|_{L^p(\Omega; \bb R^d)}^p, 
\end{split}
\end{equation}
the final inequality holding by the definition of $\mc S(A(0); G)^{-1}$ given in equation \eqref{eq:MG_sobolev_constant}. Define $\Theta\subset \bb R$ to be the set of $\lambda>0$ for which the inequality 
\begin{equation*}
	\|\bm u\|_{\mc W_A}^p
	\geq \mc S(A(0); G)^{-1}\left(\int_\Omega G(\bm u)\; \d x\right)^{p/p^*} + \lambda\|\bm u\|_{L^p(\Omega; \bb R^d)}^p
\end{equation*}
holds for all $\bm u\in \mc W$. Estimate \eqref{eq:Theta_not_empty} shows that $\Theta\neq \emptyset$. It is routine to verify that $\Theta$ is bounded above and that 
\begin{equation*}
	\lambda^*:= \sup\Theta
\end{equation*}
is in $\Theta$. Moreover, the following simple argument shows that $\lambda^*< \lambda_1(A, p)$. Let $s\in \bb S_p^{d - 1}$ satisfy $G(s) = M_G = \max\{G(t): t\in \bb S_p^{d - 1}\}$ and set $\bm u = (\varphi s_1, \ldots, \varphi s_d)$, where $\varphi\in W_0^{1, p}$ is an eigenfunction corresponding to $\lambda_1(A, p)$. For this $\bm u$ we have 
\begin{equation*}
\begin{split}
	\lambda_1(A, p)\|\varphi\|_p^p
	& = \|\bm u\|_{\mc W_A}^p\\
	& \geq \mc S(A(0); G)^{-1}\left(\int_\Omega G(\bm u)\; \d x\right)^{p/p^*} + \lambda^*\|\bm u\|_{L^p(\Omega; \bb R^d)}^p\\
	& = \mc S(A(0); G)^{-1}\left(\int_\Omega G(\bm u)\; \d x\right)^{p/p^*} + \lambda^*\|\varphi\|_{L^p(\Omega)}^p\\
	& > \lambda^*\|\varphi\|_{L^p(\Omega)}^p 
\end{split}
\end{equation*}
thereby establishing the inequality $\lambda^*< \lambda_1(A, p)$. From the definition of $\lambda^*$, if $\mu_F> \lambda^*$ then there is $\bm u\in \mc W$ for which 
\begin{equation}
\label{eq:inequality_violated}
\begin{split}
	\Phi(\bm u)
	& \leq \|\bm u\|_{\mc W_A}^p - \mu_F\|\bm u\|_{L^p(\Omega; \bb R^d)}^p\\
	& < \mc S(A(0); G)^{-1}\left(\int_\Omega G(\bm u)\; \d x\right)^{p/p^*}. 
\end{split}
\end{equation}
We fix any such $\bm u$. Lemma \ref{lemma:Phi_AF_coercive} ensures that $\Phi$ is coercive, so the strict inequality in \eqref{eq:inequality_violated} guarantees that $\int_\Omega G(\bm u)\; \d x> 0$. Defining $\bm v = \left(\int_\Omega G(\bm u)\; \d x \right)^{-1/p^*}\bm u$ we have $\int_\Omega G(\bm v)\; \d x = 1$ and, from \eqref{eq:inequality_violated} and the degree-$p$ homogeneity of $\Phi$, 
\begin{equation}
\label{eq:v_energy_small}
	\Phi(\bm v)< \mc S(A(0); G)^{-1}.
\end{equation} 
Since $x_0 = 0$ is a minimizer of $\det A$, combining Proposition \ref{prop:N(A,G)_value}, Lemma \ref{eq:constat_matrix_sobolev_constant} and Lemma \ref{eq:constat_matrix_G_sobolev_constant} yields $\mc N(A; G) =\mc S(A(0); G)$. Therefore, combining inequality \eqref{eq:v_energy_small} and Corollary \ref{coro:minimizing_solution} guarantees the existence of  a nontrivial solution to problem \eqref{eq:main_problem}. 
\end{proof}
%
\subsection{Construction of Test Function for $\gamma>p$}
This subsection is devoted to the proofs of Theorems \ref{theorem:existence_gamma_large} and \ref{theorem:boundary_minimizer_existence}. In each of these proofs we verify the energy smallness condition in \eqref{eq:energy_threshhold} by using a standard bubble-based construction to establish the existence of $\bm u\in \mc M$ for which $\Phi(\bm u)< \mc N(A; G)^{-1}$. 
\begin{proof}[Proof of Theorem \ref{theorem:existence_gamma_large}]
We assume for simplicity that $x_0 = 0\in \Omega$ in which case condition \eqref{eq:A_upper_expansion} reads
\begin{equation}
\label{eq:A_upper_expansion_origin}
	\lb A(x)\xi, \xi\rb^{p/2}
	\leq \lb A(0)\xi, \xi\rb^{p/2} + C_0|x|^\gamma|\xi|^p
\end{equation}
for all $(x, \xi)\in B(0, \delta)\times \bb R^n$. By assumptions \ref{item:A_symmetric} and \ref{item:A_positive_definite} there is an orthogonal matrix $O\in O(n)$ for which 
\begin{equation}
\label{eq:diagonalize_A(0)}
	OA(0)O^\top 
	= \diag( a_1, \ldots, a_n)
	=:D, 
\end{equation}
where $0< a_1\leq \ldots \leq a_n$ are the eigenvalues of $A(0)$. Set $P = D^{-1/2}O$ so that $PA(0)P^\top = I_n$. We also have both $a_1|y|^2\leq |P^{-1}y|^2 \leq a_n|y|^2$ for all $y\in \bb R^n$ and $|P^\top\xi|^2\leq a_1^{-1}|\xi|^2$ for all $\xi\in \bb R^n$. For any $\bm u\in \mc W$ with $\supp \bm u\subset B(0, \delta)$ we set 
\begin{equation}
\label{eq:uv_relation}
	\bm v(y) = \bm u(P^{-1}y)
	\qquad \text{ for }y\in P\Omega = \{Px: x\in\Omega\}. 
\end{equation}
Using the change of variable $y = Px$ and since $(\det P)^{-1} = \pm m_A^{1/2}$ (recall the notational convention in \eqref{eq:mA_notation}) we obtain 
\begin{equation}
\label{eq:u_WA_norm_initial_estimate}
\begin{split}
	\|\bm u\|_{\mc W_A}^p
	& \leq \sum_{j = 1}^d \int_\Omega\left(\lb A(0)\Grad u_j, \Grad u_j\rb^{p/2} + C_0|x|^\gamma |\Grad u_j|^p\right)\; \d x\\
	& \leq m_A^{1/2}\sum_{j =1 }^d\int_{P\Omega}\left(1 +C_1|y|^\gamma\right)|\Grad v_j|^p\; \d y, 
\end{split}
\end{equation}
where $C_1 = C_0 a_n^{\gamma/2}a_1^{-p/2}$. Still for $\bm u\in \mc W$ satisfying $\supp \bm u\subset B(0, \delta)$, using the same change of variable gives both of the following relations: 
\begin{equation*}
\begin{split}
	\int_\Omega F(\bm u)\; \d x 
	& = m_A^{1/2}\int_{P\Omega}F(\bm v)\; \d y\\
	\int_\Omega G(\bm u)\; \d x
	& = m_A^{1/2}\int_{P\Omega}G(\bm v)\; \d y.
\end{split}
\end{equation*}
Combining these equalities with estimate \eqref{eq:u_WA_norm_initial_estimate} we find that
\begin{equation}
\label{eq:Qu_first_pass}
\begin{split}
	Q(\bm u)
	& \leq m_A^{\frac{p}{2n}}\frac{\sum_{j = 1}^d\int_{P\Omega}\left(1 + C_1|y|^\gamma\right)|\Grad v_j|^p\; \d y - \int_{P\Omega} F(\bm v)\; \d y}{\left(\int_{P\Omega}G(\bm v)\; \d y\right)^{p/p^*}}
\end{split}
\end{equation}
whenever $\bm u$ and $\bm v$ are related via \eqref{eq:uv_relation}. If, in addition to $\bm u$ and $\bm v$ satisfying this relation, $\bm v = (s_1w, \ldots, s_d w)$ for some nonnegative function $w\in W_0^{1, p}(P\Omega)\setminus\{0\}$ having support contained in $PB(0, \delta)$ and some $s \in \bb S_p^{d- 1}$ satisfying both conditions in \eqref{eq:good_theta}, then \eqref{eq:Qu_first_pass} gives
\begin{equation}
\label{eq:quotient_on_vector_multiple}
	m_A^{-p/(2n)}M_G^{p/p^*}Q(\bm u)
	\leq \frac{\int_{P\Omega}(1 + C_1|y|^\gamma)|\Grad w|^p\; \d y - F(s)\|w\|_{L^p(P\Omega)}^p}{\|w\|_{L^{p^*}(P\Omega)}^p}. 
\end{equation}
Now, let $\eta\in C_c^\infty(PB(0, \delta))$ with $\eta \equiv 1$ in a neighborhood of the origin and, for $\epsilon>0$ small, set 
\begin{equation}
\label{eq:w_epsilon}
	w_\epsilon(y) = \eta(y)U_\epsilon(y),
\end{equation}
where $U_\epsilon = U_{\epsilon, 0}$ as given in \eqref{eq:shifted_scaled_bubble} and \eqref{eq:normalized_bubble} with $x_0 = 0$. Recalling that these functions are normalized so that $\|U_\epsilon\|_{L^{p^*}(\bb R^n)} = 1$ for all $\epsilon>0$, we have the following estimates for $w_\epsilon$, see \cite{GarciaPeral1987}: 
\begin{equation}
\label{eq:cutoff_bubble_grad_pnorm_pstar_norm}
\begin{split}
	\|\Grad w_\epsilon\|_p^p & = \mc S^{-1} + O\left(\epsilon^{\frac{n - p}{p - 1}}\right)\\
	\|w_\epsilon\|_{p^*}^p & = 1 + O(\epsilon^n)\\
\end{split}
\end{equation}
and
\begin{equation}
\label{eq:cutoff_bubble_pnorm}
	\|w_\epsilon\|_p^p 
	= \begin{cases}
	b_{n,p}\epsilon^p + O\left(\epsilon^{\frac{n - p}{p - 1}}\right)& \text{ if }n>p^2\\
	b_{n,p}\epsilon^p|\log\epsilon| + O(\epsilon^p) & \text{ if }n = p^2, 
	\end{cases}
\end{equation}
where 
\begin{equation}
\label{eq:bnp}
	b_{n,p}
	= \begin{cases}
	c_{n,p}^p\omega_{n - 1}\frac{(p - 1)\Gamma\left(\frac{n - p^2}p\right)\Gamma\left(\frac{np - n}p\right)}{p\Gamma(n -p)} & \text{ if }n > p^2\\
	c_{n,p}^p\omega_{n - 1} & \text{ if }n = p^2, 
	\end{cases}
\end{equation}
$\omega_{n - 1}= |\bb S^{n - 1}|$ is the volume of the $(n - 1)$-sphere, and $\mc S  = \mc S(n, p)$ is the sharp Sobolev constant defined in \eqref{eq:classical_sobolev_constant}. Since $\gamma> p$ we have 
\begin{equation}
\label{eq:cutoff_bubble_gamma_grad_pnrom}
\begin{split}
	\int_{\bb R^n}|y|^\gamma|\Grad w_\epsilon|^p\; \d y
 	& = \begin{cases}
	O(\epsilon^\gamma) & \text{ if }\gamma < \frac{n - p}{p - 1}\\
	O(\epsilon^\gamma|\log \epsilon|) & \text{ if } \gamma = \frac{n - p}{p - 1}\\
	O(\epsilon^{\frac{n -p}{p - 1}}) &\text{ if } \gamma> \frac{n - p}{p - 1}. 
	\end{cases}\\
	& = \begin{cases}
	\circ(\epsilon^p) & \text{ if }n>p^2\\
	O(\epsilon^p) & \text{ if }n = p^2. 
	\end{cases}
\end{split}
\end{equation}
In particular, comparing \eqref{eq:cutoff_bubble_gamma_grad_pnrom} to \eqref{eq:cutoff_bubble_pnorm} and in view of the assumption $\gamma>p$ we find that  
\begin{equation*}
	\int_{\bb R^n}|y|^\gamma|\Grad w_\epsilon|^p \; \d y
	= \circ(1)\|w_\epsilon\|_p^p
\end{equation*}
independently of whether $n>p^2$ or $n = p^2$. For $s \in \bb S_p^{d - 1}$ satisfying both conditions in \eqref{eq:good_theta} set $\bm v^\epsilon(y)= (w_\epsilon s_1, \ldots, w_\epsilon s_d)$ and set $\bm u^\epsilon(x) = \bm v^\epsilon(Px)$. Choosing $\bm u = \bm u^\epsilon$ in \eqref{eq:quotient_on_vector_multiple} gives
\begin{equation}
\label{eq:Qu_epsilon_estimate}
	m_A^{-\frac{p}{2n}}M_G^{p/p^*}Q(\bm u^\epsilon)
	\leq \frac{\int_{P\Omega}\left(1 + C_1|y|^\gamma\right)|\Grad w_\epsilon|^p\; \d y - F(s)\|w_\epsilon\|_{L^p(P\Omega)}^p}{\|w_\epsilon\|_{L^{p^*}(P\Omega)}^p}.
\end{equation}
If $n> p^2$ then combining estimates \eqref{eq:cutoff_bubble_grad_pnorm_pstar_norm}, \eqref{eq:cutoff_bubble_pnorm} and \eqref{eq:cutoff_bubble_gamma_grad_pnrom} with estimate \eqref{eq:Qu_epsilon_estimate} gives
\begin{equation}
\label{eq:n>p2_smallness_verified}
\begin{split}
	m_A^{-\frac{p}{2n}}M_G^{p/p^*}Q(\bm u^\epsilon)
	& \leq \frac{\mc S^{-1} - F(s)b_{n, p}\epsilon^p + \circ(\epsilon^p)}{1 + O(\epsilon^n)}\\
	& = \mc S^{-1} - F(s)b_{n, p}\epsilon^p + \circ(\epsilon^p)\\
	& < \mc S^{-1}, 
\end{split}
\end{equation} 
the final inequality holding for $\epsilon>0$ sufficiently small. In the case $n = p^2$ a similar computation shows that for $\epsilon>0$ sufficiently small, 
\begin{equation}
\label{eq:n=p2_smallness_verified}
\begin{split}
	m_A^{-\frac{p}{2n}}M_G^{p/p^*}Q(\bm u^\epsilon)
	& \leq \frac{\mc S^{-1} - F(s)b_{n, p}\epsilon^p|\log \epsilon| + O(\epsilon^p)}{1 + O(\epsilon^n)}\\
	& = \mc S^{-1} - F(s) b_{n, p}\epsilon^p|\log\epsilon| + O(\epsilon^p)\\
	& < \mc S^{-1}. 
\end{split}
\end{equation} 
When combined with Proposition \ref{prop:N(A,G)_value}, estimates \eqref{eq:n>p2_smallness_verified} and \eqref{eq:n=p2_smallness_verified} show that $\mc N< K$ whenever $n\geq p^2$. Corollary \ref{coro:minimizing_solution} now ensures that problem \eqref{eq:main_problem} has a nontrivial weak solution. 
\end{proof}
\begin{proof}[Proof of Theorem \ref{theorem:boundary_minimizer_existence}]
As in the proof of Theorem \ref{theorem:existence_gamma_large}, we assume for simplicity that $x_0 = 0$ and we define $P = D^{-1/2}O$, where $O\in O(n)$ and $D$ are as in \eqref{eq:diagonalize_A(0)}. If $0\leq w\in W_0^{1, p}(P\Omega)\setminus\{0\}$ is a nonnegative function with support contained in $P(B(0, \delta)\cap \Omega)$ and if $s\in \bb S_p^{d - 1}$ satisfies both conditions in \eqref{eq:good_theta} then for $\bm u(x) = (s_1w(Px), \ldots, s_dw(Px))$, we still have estimate \eqref{eq:quotient_on_vector_multiple}. It is routine to check that the interior $\theta$-singularity assumption on $\bdy\Omega$ at $x_0 = 0$ guarantees the interior $\theta$-singularity of $\bdy P\Omega$ at $y_0 = Px_0 = 0$. Indeed, if $r>0$ and if $(x_i)_{i = 1}^\infty\subset \Omega$ with $x_i\to 0$ and $B(x_i, r|x_i|^\theta)\subset\Omega$, then for $y_i = Px_i$ we have both $|y_i|\to 0$ and $B(y_i, r'|y_i|^\theta)\subset P\Omega$, where $r' = a_n^{-1/2}a_1^{\theta/2}r$. For $x_i$ and $y_i$ as such, let $\eta\in C_c^\infty(\bb R^n)$ satisfy $0\leq \eta\leq 1$, $\eta\equiv 1$ in $B(0, r'/2)$ and $\eta\equiv 0$ on $\bb R^n\setminus B(0, r')$ and consider the test function 
\begin{equation*}
	\varphi^i(y)
	= \eta\left(\frac{y - y_i}{\epsilon_i^\theta}\right)U_{\epsilon_i^\sigma}(y - y_i), 
\end{equation*}
where $\epsilon_i = |y_i|$, where $U_\epsilon = U_{\epsilon, 0}$ is given in \eqref{eq:shifted_scaled_bubble} and \eqref{eq:normalized_bubble} with $x_0 = 0$ and where $\sigma$ satisfies
\begin{equation}
\label{eq:sigma}
	\max\left\{\theta p^*, \frac{\theta p(n -p)}{n - p^2}\right\}
	< p\sigma 
	< \gamma. 
\end{equation}
In particular $\sigma> \theta$. With $w_\epsilon$ as in \eqref{eq:w_epsilon} we have
\begin{equation*}
	\varphi^i(y) = \epsilon_i^{-\theta(n -p)/p}w_{\epsilon_i^{\sigma - \theta}}\left(\frac{ y - y_i}{\epsilon_i^\theta}\right),
\end{equation*}
so estimates \eqref{eq:cutoff_bubble_grad_pnorm_pstar_norm} and \eqref{eq:cutoff_bubble_pnorm} yield the following estimates 
\begin{equation}
\label{eq:phii_norm_estimates}
\begin{split}
	\|\Grad \varphi^i\|_{L^p(\bb R^n)}^p 
	& = \|\Grad w_{\epsilon_i^{\sigma - \theta}}\|_{L^p(\bb R^n)}^p = \mc S^{-1} + O(\epsilon_i^{(\sigma - \theta)(n -p)/(p - 1)})\\
	\|\varphi^i\|_{L^{p^*}(\bb R^n)}^p 
	& = \|w_{\epsilon_i^{\sigma - \theta}}\|_{L^{p^*}(\bb R^n)}^p = 1 + O(\epsilon_i^{n(\sigma - \theta)})\\
	\|\varphi^i\|_{L^p(\bb R^n)}^p 
	& = \epsilon_i^{p\theta}\|w_{\epsilon_i^{\sigma - \theta}}\|_{L^p(\bb R^n)}^p = 
	b_{n, p}\epsilon_i^{p\sigma} + O(\epsilon_i^{p\theta + (\sigma - \theta)(n - p)/(p - 1)}),
\end{split}
\end{equation} 
where $b_{n,p}$ is as in \eqref{eq:bnp}.
Moreover, since $\supp\varphi^i\subset B(y^i, r'\epsilon_i^\theta)$ and since $\theta\geq 1$, for any $y\in \supp\varphi^i$ we have
\begin{equation*}
	|y|\leq |y - y_i| +|y_i|\leq r'\epsilon_i^\theta + \epsilon_i\leq C\epsilon_i. 
\end{equation*}
Therefore, the first estimate in \eqref{eq:phii_norm_estimates} gives
\begin{equation}
\label{eq:gamma_grad}
	\int_{\bb R^n}|y|^\gamma |\Grad \varphi^i|^p\; \d y
	\leq C\epsilon_i^\gamma\|\Grad \varphi^i\|_{L^p(\bb R^n)}^p
	= O(\epsilon_i^\gamma). 
\end{equation}
For any small $r>0$, by choosing $i$ sufficiently large we have $\supp(\varphi^i\circ P)\subset B(0, \delta)$, where $\delta>0$ is as in the hypotheses of the theorem, so choosing $s\in \bb S_p^{d -1}$ for which both of conditions \eqref{eq:good_theta} hold, and setting $\bm u^i(x) = (s_1\varphi^i(Px), \ldots, s_d\varphi^i(Px))$, we may apply estimate \eqref{eq:quotient_on_vector_multiple} to obtain 
\begin{equation*}
\begin{split}
	m_A^{-p/(2n)}& M_G^{p/p^*}Q(\bm u^i)\\
	& = \frac{\int_{P\Omega}(1 + C_1|y|^\gamma)|\Grad \varphi^i|^p\; \d y - F(s)\|\varphi^i\|_{L^p(\bb R^n)}^p}{\|\varphi^i\|_{L^{p^*}(\bb R^n)}^p}\\
	& = \frac{\mc S^{-1} - F(s)b_{n,p}\epsilon^{p\sigma} + O(\epsilon^{(\sigma - \theta)(n - p)/(p - 1)}) + O(\epsilon_i^{\gamma})}{1 + O(\epsilon_i^{n(\sigma - \theta)})}\\
	& = \mc S^{-1} - F(s)b_{n,p}\epsilon^{p\sigma} + O(\epsilon^{(\sigma - \theta)(n - p)/(p - 1)}) + O(\epsilon_i^{\gamma}) + O(\epsilon_i^{n(\sigma - \theta)}), 
\end{split}
\end{equation*}
where the second equality is obtained with the aid of estimates \eqref{eq:phii_norm_estimates} and \eqref{eq:gamma_grad}. Since $\sigma$ satisfes \eqref{eq:sigma}, we find that $p\sigma< \min\{\gamma, \frac{(\sigma -\theta)(n - p)}{p - 1}, n(\sigma - \theta)\}$, so for $i$ sufficiently large we obtain 
\begin{equation*}
	Q(\bm u^i)
	< m_A^{p/(2n)}M_G^{-p/p^*}\mc S^{-1}
	= \mc N^{-1}, 
\end{equation*}
the equality holding by Proposition \ref{prop:N(A,G)_value}. Finally, Corollary \ref{coro:minimizing_solution} now implies that problem \eqref{eq:main_problem} admits a nontrivial solution. 
\end{proof}
\section{Proof of the Non-Existence Result}
\label{s:non_existence}
For $C^2$ solutions to problem \eqref{eq:potential_form_system} below, the following Pohozaev identity follows from the general variational identity of Pucci and Serrin \cite{PucciSerrin1986}. Because of the strict convexity of the map $\xi\mapsto \mc F(z, x, \xi) = \frac 1p\lb A(x)\xi, \xi\rb^{p/2} - H(z)$ the results of \cite{Degiovanni2003} guarantee that the identity also holds under relaxed smoothness assumptions. 
\begin{lemma}
\label{lemma:pohozaev_general}
Let $p\in (1, n)$, let $\Omega\subset \bb R^n$ be a bounded domain with $C^1$ boundary and let $x_0\in \overline\Omega$. Suppose $A\in C^1(\overline \Omega\setminus\{x_0\})$ satisfies \ref{item:A_symmetric}, \ref{item:A_positive_definite} and that $b_{ij}(x) = (x - x^0)\cdot \Grad a_{ij}(x)$ extends continuously to $x_0$ for all $i,j$. If $\bm u\in \mc W\cap C^1(\overline\Omega; \bb R^d)$ is a weak solution to 
\begin{equation}
\label{eq:potential_form_system}
\begin{cases}
	-L_{A, p}\bm u = h(\bm u) & \text{ in }\Omega\\
	\bm u = 0 & \text{ on }\bdy \Omega
\end{cases}
\end{equation}
where $h = \Grad H$ for some $H\in C^1(\bb R^d)$ satisfying $H(0) = 0$ then 
\begin{equation}
\label{eq:general_pohozaev}
\begin{split}
	np\int_\Omega H(\bm u)\; \d x&  - (n - p)\int_\Omega h(\bm u)\cdot \bm u\; \d x\\
	& = (p - 1)\int_{\bdy\Omega}|\Grad \bm u|_p^p\lb A(x)\nu, \nu\rb^{p/2}\lb x - x_0, \nu\rb \; \d S_x\\
	&  + \frac p 2\sum_{j = 1}^d\int_\Omega\lb A(x)\Grad u_j, \Grad u_j\rb^{(p - 2)/2}\lb B(x)\Grad u_j, \Grad u_j\rb\; \d x, 
\end{split}
\end{equation}
where $B(x) = (b_{ij}(x))$. 
\end{lemma}
\begin{coro}
Under the hypotheses of Lemma \ref{eq:potential_form_system} with the additional assumption that $H(s) = \frac 1p F(s) + \frac{1}{p^*}G(s)$ for some functions $F$ and $G$ satisfying \ref{item:F_homogeneous} and \ref{item:G_positively_homogeneous} respectively, 
\begin{equation}
\label{eq:FG_pohozaev}
\begin{split}
	p\int_\Omega F(\bm u)\; \d x
	& = (p -1)\int_{\bdy \Omega}|\Grad \bm u|_p^p\lb A(x)\nu, \nu\rb^{p/2}\lb x - x^0, \nu\rb \; \d S_x\\
	& + \frac p2\sum_{j = 1}^d\int_\Omega \lb A(x)\Grad u_j,  \Grad u_j\rb^{(p - 2)/2}\lb B(x)\Grad u_j, \Grad u_j\rb \; \d x. 
\end{split}
\end{equation}
\end{coro}
\begin{proof}
From the homogeneity assumptions on $F$ and $G$, with $h = \Grad H = \frac 1p\Grad F + \frac 1{p^*}\Grad G$, Lemma \ref{lemma:properties_homogeneous_functions} \ref{item:homogeneous_grad_dot} gives $h(\bm u)\cdot \bm u = F(\bm u) + G(\bm u)$. Using this in equation \eqref{eq:general_pohozaev} yields equation \eqref{eq:FG_pohozaev}. 
\end{proof}
\begin{proof}[Proof of Theorem \ref{theorem:pozohaev_nonexistence}]
Suppose $\bm u\in \mc W\cap C^1(\overline\Omega; \bb R^d)$ is a nontrivial weak solution to problem \eqref{eq:main_problem}. The assumption that $\Omega$ is star-shaped with respect to $x_0$ ensures that $\lb x - x_0, \nu\rb\geq 0$ for all $x\in \bdy \Omega$ and assumption \ref{item:A_positive_definite} gives $\lb A(x)\nu, \nu\rb \geq \tau > 0$ for all $x\in \bdy \Omega$, where $\nu = \nu(x)$ is the outward unit normal vector at $x\in \bdy \Omega$. Thus the boundary integral on the right-hand side of \eqref{eq:FG_pohozaev} is non-negative. Next, we observe that condition \eqref{eq:B(x)_lower_bound} is equivalent to 
\begin{equation*}
	\frac p2 \lb A(x)\xi, \xi\rb^{(p - 2)/2}\lb B(x)\xi, \xi\rb
	\geq \gamma C_0 |x - x_0|^\gamma |\xi|^p
	\qquad \text{ for all }(x, \xi)\in \Omega\times \bb R^n,
\end{equation*}
so from identity \eqref{eq:FG_pohozaev}, from Lemma \ref{lemma:properties_homogeneous_functions} \ref{item:homogeneous_upper_bound}, and from the sharp Hardy-Sobolev inequality \eqref{eq:hardy_sobolev} we have 
\begin{equation*}
\begin{split}
	pM_F\|\bm u\|_{L^p(\Omega;\bb R^d)}^p
	& \geq p\int_\Omega F(\bm u)\;\d x\\
	& \geq \frac p2\sum_{j = 1}^d\int_\Omega \lb A(x)\Grad u_j, \Grad u_j\rb^{(p - 2)/2}\lb B(x)\Grad u_j, \Grad u_j\rb\; \d x\\
	& \geq \gamma C_0\sum_{j = 1}^d\int_\Omega|x - x_0|^\gamma |\Grad u_j|^p\; \d x\\
	& \geq \frac{\gamma C_0}{K_0^p}\|\bm u\|_{L^p(\Omega; \bb R^d)}^p
\end{split}
\end{equation*}
Since $\bm u\not\equiv 0$ we deduce that $M_F\geq \lambda_*:= \frac{\gamma C_0}{pK_0^p}$. 
\end{proof}
\section{Appendix}
\label{s:appendix}
\begin{lemma}
\label{lemma:simple_construction}
Let $\Omega \subset \bb R^n$ be open. If $F\in C^0(\bb R^d)$ is homogeneous of degree $p$ and if $M_F> 0$ then there is $\bm u\in W_0^{1, p}(\Omega; \bb R^d)$ such that $\int_\Omega F(\bm u) \; \d x> 0$. 
\end{lemma}
\begin{proof}[Proof of Lemma \ref{lemma:simple_construction}]
By the continuity of $F$ and the assumption $M_F>0$ there is $\sigma\in \bb S_p^{d - 1}$ for which $F(\sigma) = M_F$ and there is $\epsilon>0$ such that $F(s)> M_F/2$ for all $s\in B(\sigma, \epsilon)$. Let $\eta\in C_c^\infty(\bb R)$ satisfy
\begin{equation*}
	\eta(t) = \begin{cases}
	0 & \text{ if } t\leq 0\\
	1 & \text{ if }t\geq 1. 
	\end{cases}
\end{equation*}
For $B(x_0, r)\subset \Omega$, define $\bm u\in C_c^\infty(\Omega; \bb R^d)$ by 
\begin{equation*}
	\bm u(x) 
	= \left( 1- \eta\left(\frac{|x - x_0|^2}{r^2}\right)\right)\sigma. 
\end{equation*}
By homogenity of $F$, for any $x\in B(x_0, r)$ we have
\begin{equation*}
	F(\bm u(x))
	= \left( 1- \eta\left(\frac{|x - x_0|^2}{r^2}\right)\right)^pM_F
	\geq 0. 
\end{equation*}
Moreover, there is $\delta\in(0, r)$ such that $\bm u(x)\in B(\sigma, \epsilon)$ whenever $x\in B(x_0, \delta)$. For any such $\delta$, 
\begin{equation*}
	\int_\Omega F(\bm u)\; \d x
	\geq \int_{B(x_0, \delta)}F(\bm u)\; \d x
	\geq \frac{M_F}2|B(x_0, \delta)|
	> 0. 
\end{equation*}
\end{proof}
\begin{lemma}
\label{lemma:elementary_Lp}
Let $\Omega \subset \bb R^n$ be an open set. For each $p\geq 1$ there is a constant $C = C(p)>0$ such that for all $\epsilon\in (0, 1)$ and all non-negative functions $f, h\in L^p(\Omega)$, 
\begin{equation*}
	\int_\Omega (f + h)^p\; \d x
	\leq (1 + \epsilon)\|f\|_p^p + \frac{C(p)}{\epsilon^p}\|h\|_p^p. 
\end{equation*}
\end{lemma}
\begin{proof}
For any $t\in [0, 1]$, the convexity of $t\mapsto (1 + t)^p$ ensures that $(1+ t)^p \leq 1 + (2^p-1)t$ for all $t\in [0, 1]$ while for $t\geq 1$ the Mean Value Theorem gives
\begin{equation*}
\begin{split}
	(1 + t)^p - 1^p
	\leq pt(2t)^{p - 1}
	= 2^{p - 1}pt^p. 
\end{split}
\end{equation*}
Combining these two estimates yields
\begin{equation*}
	(1 + t)^p \leq \begin{cases}
	1 + (2^p - 1)t & \text{ if }t\in [0, 1]\\
	1 + 2^{p - 1}pt^p & \text{ if }t\geq 1.
	\end{cases}
\end{equation*}
From this estimate we obtain $(1 + t)^p\leq 1+ 2^p t + 4^p t^p$ for all $t\geq0$ and hence 
\begin{equation}
\label{eq:(a + b)^p}
	(a + b)^p
	\leq a^p + 2^pa^{p - 1}b + 4^p b^p
	\qquad \text{ for all }a, b\geq 0.  
\end{equation} 
Now if $0\leq f, g\in L^p(\Omega)$ and if $\delta\in (0, 1)$ then applying \eqref{eq:(a + b)^p}, H\"older's inequality and Young's inequality gives
\begin{equation*}
\begin{split}
	\int_\Omega (f + h)^p\; \d x
	& \leq \|f\|_p^p + 2^p\int_\Omega f^{p-1} h\; \d x + 4^p\|h\|_p^p\\
	& \leq \|f\|_p^p + 2^p\left(\frac{\delta^{p'}}{p'}\|f\|_p^p + \frac{1}{p\delta^p}\|h\|_p^p\right) + 4^p\|h\|_p^p\\
	& \leq (1 + 2^p \delta^{p'})\|f\|_p^p + \left(4^p + 2^p\delta^{-p}\right)\|h\|_p^p.  
\end{split}
\end{equation*}
Given $\epsilon\in (0, 1)$ we apply the previous inequality with $\delta = 2^{-\frac p{p'}}\epsilon^{\frac 1{p'}}$ to obtain 
\begin{equation*}
\begin{split}
	\int_\Omega (f + h)^p\; \d x
	& \leq (1 + \epsilon)\|f\|_p^p + \left(4^p + 2^p\left(2^{\frac p{p'}}\epsilon^{-\frac 1{p'}}\right)^p\right)\|h\|_p^p\\
	& \leq (1+ \epsilon)\|f\|_p^p + 4^{p^2+ 1}\epsilon^{-p}\|h\|_p^p. 
\end{split}
\end{equation*}
\end{proof}
\begin{proof}[Proof of Lemma \ref{lemma:good_minimizing_sequence}]
Define $X = \{\bm u\in \mc W_A:\Psi(\bm u)\geq 1/4\}$. The continuity of $\Psi$ on $\mc W_A$ ensures that  $(X, \|\cdot\|_{\mc W_A})$ is a complete metric space. Let $(\bm v^k)_{k = 1}^\infty\subset \mc M$ be a minimizing sequence for $K^{-1}$ for which 
\begin{equation*}
	K^{-1} + \frac 1{k^2} 
	\geq \Phi(\bm v^k)
	= Q(\bm v^k). 
\end{equation*}
For each $k$ we apply the Ekeland Variational Principle to obtain $\bm w^k\in X$ that strictly minimizes the functional 
\begin{equation*}
	Q_k(\bm \varphi) 
	= Q(\bm \varphi) + \frac 1 k\|\bm w^k - \bm \varphi\|_{\mc W_A}
\end{equation*}
over $\bm \varphi\in X$ and for which both $Q(\bm w^k)\leq Q(\bm v^k)$ and 
\begin{equation}
\label{eq:ekelands_sequence_close}
	\|\bm v^k - \bm w^k\|_{\mc W_A}\leq \frac 1 k. 
\end{equation}
By continuity of $\Psi:\mc W_A\to \bb R$, the assumption $\Psi(\bm v^k) = 1$ for all $k$, and \eqref{eq:ekelands_sequence_close} we obtain 
\begin{equation}
\label{eq:Psi(wk)_bounds}
	\frac 78 \leq \Psi(\bm w^k) \leq \frac98
\end{equation}
whenever $k$ is sufficiently large. Using the continuity of $\Psi$ once more, together with the fact that $\bm w^k$ strictly minimizes $Q_k$, we find that there is $\delta>0$ such that 
\begin{equation*}
\begin{split}
	Q(\bm w^k)
	& = Q_k(\bm w^k)\\
	& \leq Q_k(\bm w^k + \bm \eta)\\
	& = Q(\bm w^k + \bm \eta) + k^{-1}\|\bm \eta\|_{\mc W_A}
\end{split}
\end{equation*}
whenever $\|\bm \eta\|_{\mc W_A}< \delta$. Therefore, 
\begin{equation*}
	\|Q'(\bm w_k)\|_{\mc W_A'}
	= \limsup_{0\neq \|\bm \eta\|_{\mc W_A}\to 0}\frac{Q(\bm w^k) - Q(\bm w^k + \bm \eta)}{\|\bm \eta\|_{\mc W_A}}
	\leq k^{-1}
	= \circ(1). 
\end{equation*}
Finally, for $c_k = \Psi(\bm w^k)^{-1/p^*}$ by \eqref{eq:Psi(wk)_bounds} there is a constant $C = C(n, p)\geq 1$ such that for $k$ sufficiently large,
\begin{equation}
\label{eq:ck_bounds}
	C^{-1}\leq c_k\leq C.
\end{equation}
Defining $\bm u^k = c_k\bm w^k$, the $p^*$-homogeneity of $\Psi$ gives $\Psi(\bm u^k) = 1$ for all $k$. Since $Q$ is homogeneous of degree $0$ we get 
\begin{equation*}
	Q(\bm u^k)
	= Q(\bm w^k)
	\leq Q(\bm v^k)
	\leq K^{-1} + \frac 1{k^2}. 
\end{equation*}
Moreover since $Q'(c\bm u) = c^{-1}Q'(\bm u)$ for $c>0$ and from \eqref{eq:ck_bounds} we have 
\begin{equation*}
	\|Q'(\bm u^k)\|_{\mc W_A'}
	= c_k^{-1}\|Q'(\bm w^k)\|_{\mc W_A'}
	\leq C\|Q'(\bm w^k)\|_{\mc W_A'}
	= \circ(1). 
\end{equation*}
Thus $(\bm u^k)_{k= 1}^\infty\subset \mc M$ is a minimizing sequence for $K^{-1}$ for which $Q'(\bm u^k)\to 0$ in $\mc W_A'$. We proceed to show that the condition $Q'(\bm u^k)\to 0$ in $\mc W_A'$ implies condition \eqref{eq:some_derivatives_vanishing}. For $(\bm u^k)_{k = 1}^\infty\subset \mc M$ constructed as above we have
\begin{equation*}
\begin{split}
	\circ(1)
	& = Q'(\bm u^k)\\
	& = \Phi'(\bm u^k) - \frac p{p^*}\Phi(\bm u^k)\Psi'(\bm u^k)\\
	& = \Phi'(\bm u^k) - \frac p{p^*}K^{-1}\Psi'(\bm u^k) + \circ(1)\Psi'(\bm u^k). 
\end{split}
\end{equation*}
Thus, to complete the proof of the lemma we only need to show that $\sup_k\|\Psi'(\bm u^k)\|_{\mc W_A'}< \infty$. We proceed to verify this now. Since $G\in C^1(\bb R^d)$, item \ref{item:derivatives_homogeneous} of Lemma \ref{lemma:properties_homogeneous_functions} guarantees that $|\Grad G|_2$ is $(p^*- 1)$-homogeneous, so item \ref{item:homogeneous_upper_bound} of Lemma \ref{lemma:properties_homogeneous_functions} gives $|\Grad G(s)|_2\leq \mu |s|^{p^* - 1}_{p^* - 1}$ for all $s\in \bb R^d$, where $\mu$ is defined by $\mu:= \max\{|\Grad G(s)|_2: s\in \bb S_{p^* - 1}^{d - 1}\}$. Therefore, for any $\bm u, \bm \varphi\in \mc W_A$, 
\begin{equation*}
\begin{split}
	\abs{\lb \Psi'(\bm u), \bm \varphi\rb}
	& \leq \int_\Omega |\Grad G(\bm u)|_2|\bm \varphi|_2\; \d x\\
	& \leq \mu\int_\Omega |\bm u|_{p^* - 1}^{p^* - 1}|\bm \varphi|_2\; \d x\\
	& \leq \mu\left(\int_\Omega |\bm u|_{p^* - 1}^{p^*}\; \d x\right)^{(p^* - 1)/p^*}\left(\int_\Omega |\bm \varphi|_2^{p^*}\; \d x\right)^{1/p^*}\\
	& \leq C\mu\sum_{j = 1}^d\|u_j\|_{L^{p^*}}^{p^* - 1}\|\bm \varphi\|_{\mc W_A}. 
\end{split}
\end{equation*}
The coercivity of $\Phi$ ensures that $\bm u^k$ is bounded in $L^{p^*}(\Omega; \bb R^d)$ so applying the previous estimate with $\bm u = \bm u^k$ and arbitrary $\bm \varphi\in \mc W_A$ yields
\begin{equation*}
	\sup_k \|\Phi'(\bm u^k)\|_{\mc W_A'} 
	\leq C\sup_k\|\bm u^k\|_{L^{p^*}(\Omega; \bb R^d)}^{p^* - 1}
	< \infty 
\end{equation*}
thereby completing the proof. 
\end{proof}
%
\subsection{Regularity}
\label{ss:regularity}
\begin{proof}[Proof of Proposition \ref{prop:regularity}]
Let $\bm u\in \mc W$ be a weak solution to problem \eqref{eq:main_problem}. First we show that $\bm u\in L^t(\Omega; \bb R^d)$ for any $t\in (p^*, \infty)$ by adapting the argument of Proposition 1.2 of \cite{GueddaVeron1989} to the vector-valued case. For $k\in \bb N$ large and for $q\in (p, \infty)$, define $\zeta: \bb R\to \bb R$ by
\begin{equation*}
	\zeta(r)
	= \begin{cases}
	\sign r|r|^{q/p} & \text{ if }|r|\leq k\\
	\sign r\left(k^{q/p} + \frac qp k^{\frac qp - 1}(|r| - k)\right) & \text{ if }|r|> k. 
	\end{cases}
\end{equation*}
The function $\eta\in C^1(\bb R)$ defined by $\eta(r) = \int_0^r(\zeta'(s))^p\; \d s$ has bounded derivative and satisfies $\eta(0) = 0$, so $\eta(\bm u):= (\eta(u_1), \ldots, \eta(u_d))\in \mc W$ is a valid test function for the weak formulation of problem \eqref{eq:main_problem}. Setting $h = f + g\in C^0(\bb R^d; \bb R^d)$, testing problem \eqref{eq:main_problem} against $\eta(\bm u)$, then using the Sobolev inequality and performing standard computations gives
\begin{equation}
\label{eq:regularity_initial_test}
\begin{split}
	\sum_{j =1 }^d \int_\Omega h_j(\bm u)\eta(u_j)\; \d x
	& \geq \frac{\tau^{p/2}}{\mc S}\sum_{j = 1}^d\|\zeta(|u_j|)\|_{p^*}^p\\
	& \geq \frac{\tau^{p/2}}{2^{dp}\mc S}\left(\sum_{j = 1}^d\|\zeta(|u_j|)\|_{p^*}\right)^p\\
	& \geq \frac{\tau^{p/2}}{2^{dp}\mc S}\big\|\sum_{j = 1}^d\zeta(|u_j|)\big\|_{p^*}^p\\
	& \geq \frac{\tau^{p/2}}{2^{dp}C_0^p\mc S}\|\zeta(|\bm u|)\|_{p^*}^p, 
\end{split}
\end{equation}
where $\mc S = \mc S(n, p)$ is the sharp Sobolev constant defined in equation \eqref{eq:classical_sobolev_constant} and $C_0 = C_0(d,p, q)$ is any constant for which the inequality $\zeta(|z|)\leq C_0\sum_{j = 1}^d\zeta(|z_j|)$ holds for all $z\in \bb R^d$. From the continuity and the homogeneity of $f$ and $g$ we have $C>0$ such that $|h(z)|\leq C(1 + |z|^{p^* - p})|z|^{p - 1}$ for all $z\in \bb R^d$. Since, in addition, the inequality 
\begin{equation*}
	|r|^{p - 1}|\eta(r)| \leq 2\left(\frac qp\right)^p\zeta(|r|)^p
\end{equation*}
holds for all $r\in \bb R$, for any $m\in \bb N$, the $j^{\text{th}}$ summand on the left-hand side of \eqref{eq:regularity_initial_test} is bounded above as follows: 
\begin{equation}
\label{eq:regularity_bound_jth_summand_above}
\begin{split}
	\int_\Omega & h_j(\bm u)\eta(u_j)\; \d x\\
	& \leq C\int_\Omega (1 + |\bm u|^{p^* - p})|\bm u|^{p - 1}\eta(|\bm u|)\; \d x\\
	& \leq C\int_\Omega (1 + |\bm u|^{p^* - p})\zeta(|\bm u|)^p\; \d x\\
	& \leq C\left(\|1+ |\bm u|^{p^* - p}\|_{L^{n/p}(\Omega_m)}\|\zeta(|\bm u|)\|_{p^*}^p + (1 + m^{p^* - p})\|\zeta(|\bm u|)\|_p^p\right), 
\end{split}
\end{equation}
where $C$ depends on $d$, $p$, $q$ and the maximum values of $|f|$ and $|g|$ over $\bb S^{d - 1}$ and where $\Omega_m = \{x\in \Omega: |\bm u|\geq m\}$. Combining \eqref{eq:regularity_initial_test} and \eqref{eq:regularity_bound_jth_summand_above} then choosing $m$ sufficiently large and depending on $n$, $d$, $p$, $q$, $\tau$, the maximum values of $|f|$ and $|g|$ over $\bb S^{d - 1}$, and the distribution function of $|\bm u|$ gives
\begin{equation}
\label{eq:integrability_boosting_estimate}
	\|\zeta(|\bm u|)\|_{p^*}^p \leq C\|\zeta(|\bm u)\|_p^p. 
\end{equation}
If $q$ is any exponent for which $|\bm u|\in L^q(\Omega)$ then letting $k\to \infty$ in \eqref{eq:integrability_boosting_estimate} gives
\begin{equation}
\label{eq:iteration_estimate}
	\|\bm u\|_{qp^*/p}^q\leq C\|\bm u\|_q^q. 
\end{equation}
Given $t\in (p^*, \infty)$, one may iterate estimate \eqref{eq:iteration_estimate} finitely many times starting with $q = p^*$ to obtain $\bm u\in L^t(\Omega; \bb R^d)$. \\

With the improved integrability of $\bm u$ in hand, for each $j \in \{1, \ldots, d\}$ we have $h_j\circ \bm u\in L^t(\Omega)$ for some $t> n/p$ so since $u_j\in W_0^{1, p}(\Omega)$ weakly solves $-L_{A, p}u_j = h_j\circ \bm u$ in $\Omega$, the standard argument based on Moser's iteration gives $u_j\in L^\infty(\Omega)$. The boundedness of each $u_j$ gives $h_j\circ \bm u\in L^\infty(\Omega)$, so from the assumption that $\bdy\Omega\in C^{1, \alpha}$ and the assumption that the entries of $A$ are in $C^\alpha(\overline\Omega)$, we may apply Theorem 1 of \cite{Lieberman1988} to obtain $u_j\in C^{1, \alpha}(\overline\Omega)$ for each $j \in \{1, \ldots, d\}$.
\end{proof}
%
\newcommand{\etalchar}[1]{$^{#1}$}

\end{document}